\documentclass[12pt, colorinlistoftodos]{amsart}
\usepackage{a4wide}
\usepackage{amssymb}
\usepackage{amsthm}
\usepackage{amsmath}
\usepackage{amscd}
\usepackage{verbatim}
\usepackage{bm}
\usepackage[mathscr]{eucal}
\usepackage[all]{xy}
\usepackage{todonotes}
\usepackage[backref=true,giveninits=true,doi=false,url=false,isbn=false,backend=biber]{biblatex}
\usepackage{hyperref}
\usepackage{dsfont, stmaryrd}
\usepackage{mathrsfs}
\usepackage{mathtools}
\allowdisplaybreaks[4]

\numberwithin{equation}{section}

\theoremstyle{plain}
\newtheorem{theorem}{Theorem}[section]

\newtheorem{lemma}[theorem]{Lemma}
\newtheorem{proposition}[theorem]{Proposition}

\theoremstyle{definition}
\newtheorem{definition}[theorem]{Definition}

\newtheorem{remark}[theorem]{Remark}

\theoremstyle{remark}

\newcommand{\A}{\mathbb{A}}
\newcommand{\R}{\mathbb{R}}
\newcommand{\Q}{\mathbb{Q}}
\newcommand{\Z}{\mathbb{Z}}
\newcommand{\N}{\mathbb{N}}
\newcommand{\h}{\mathbb{H}}

\newcommand{\C}{\mathbb{C}}

\newcommand{\e}{\mathbf{e}}
\newcommand{\G}{\mathbb{G}}

\newcommand{\Sc}{\mathcal{S}}
\newcommand{\Fc}{\mathcal{F}}

\newcommand{\Oc}{\mathcal{O}}
\newcommand{\Dc}{\mathcal{D}}

\newcommand{\Nm}{\mathrm{Nm}}
\newcommand{\tr}{\mathrm{Tr}}

\newcommand{\erf}{\mathrm{erf}}
\newcommand{\erfc}{\mathrm{erfc}}
\newcommand{\Cl}{\mathrm{Cl}}

\newcommand{\GL}{\mathrm{GL}}
\newcommand{\SL}{\mathrm{SL}}
\newcommand{\GO}{\mathrm{GO}}
\newcommand{\SO}{\mathrm{SO}}
\newcommand{\Spin}{\mathrm{Spin}}
\newcommand{\Gspin}{\mathrm{Gspin}}

\newcommand{\Gal}{\mathrm{Gal}}
\newcommand{\Ind}{\mathrm{Ind}}
\newcommand{\Res}{\mathrm{Res}}
\newcommand{\cha}{\mathrm{Char}}
\newcommand{\sgn}{\mathrm{sgn}}
\newcommand{\frob}{\mathrm{Frob}}
\newcommand{\bfrob}{\overline{\mathrm{Frob}}}

\newcommand{\vol}{\mathrm{vol}}

\newcommand{\Art}{\mathrm{Art}}
\newcommand{\reg}{\mathrm{reg}}

\newcommand{\lb}{\left(}
\newcommand{\rb}{\right)}
\newcommand{\hb}{\mathbf{h}}
\newcommand{\gb}{\mathbf{g}}

\newcommand{\wk}{\widetilde{K}}
\newcommand{\wf}{\widetilde{F}}
\newcommand{\wI}{\widetilde{I}}
\newcommand{\wW}{\widetilde{W}}
\newcommand{\wrho}{\widetilde{\rho}}
\newcommand{\fu}{\mathfrak{u}}
\newcommand{\fp}{\mathfrak{p}}
\newcommand{\fq}{\mathfrak{q}}
\newcommand{\mfp}{\mathfrak{P}}

\newcommand{\pmat}[4]{\begin{pmatrix}
                 #1 & #2\\
                 #3 & #4
\end{pmatrix}}

\begin{document}
\title[Twisted Siegel-Weil formulas]{Twisted Siegel-Weil formulas for $\GL_2$ over non-Galois quartic CM fields}
\author[M.~Zhang]{Mingkuan Zhang}
\address{Fachbereich Mathematik,
Technische Universit\"at Darmstadt, Schlossgartenstrasse 7, D--64289
Darmstadt, Germany}
\email{mzhang@mathematik.tu-darmstadt.de}
\subjclass[2020]{}
\thanks{
}

\begin{abstract}
We establish twisted Siegel-Weil formulas for non-Galois quartic CM fields, identifying the twisted theta integral against a quadratic character with the Doi-Naganuma lift of Hecke's integral. 
This implies the base change of Jacquet-Langlands correspondence for certain Hecke characters from $\Q$ to real quadratic fields via Doi-Naganuma lift is an isomorphism. 
As an application, we prove that the twisted CM values of Borcherds forms are algebraic multiple of logarithm of units, explicitly described by the Fourier coefficients of twisted theta integrals.
\end{abstract}

\date{\today}
\maketitle

\makeatletter
\providecommand\@dotsep{5}
\def\listtodoname{List of Todos}
\def\listoftodos{\@starttoc{tdo}\listtodoname}
\makeatother

\tableofcontents

\section{Introduction}
Siegel-Weil formulas establish a relation between integral of theta series and Eisenstein series, which is discovered by Siegel \cite{Siegel35} and proved in great generality via Weil representation by Weil \cite{Weil65}.
In the language of global theta correspondence \cite{Howe79,GT16} for the symplectic-orthogonal dual reductive pair, it shows the global theta lift of the trivial representation equals an automorphic representation generated by the values of Eisenstein series within Weil's convergent range. 
There are also many generalizations concern Laurent expansions of Eisenstein series, to other dual reductive pairs, and in geometric or arithmetic settings \cite{Chao23,KR94,Ichino04,GQT14}.

Instead of the trivial representation, one can also consider the theta lift of automorphic characters. However, this is not well understood yet. 
For the dual pair $(\widetilde{\mathrm{SL}}_2, \mathrm{O}(3))$, Waldspurger \cite{Walds84,Walds91} studied the local theta correspondence and obtained explicit information about the associated representations. 
Matching Whittaker coefficients, Snitz \cite{Snitz07} considered the global theta lifts of automorphic characters $\chi \neq 1$ of $\mathrm{O}(3)$ and expressed them explicitly in terms of theta series arising from various $\mathrm{O}(1)$'s as a weighted orbital integral at the local level.
This provides a complement to the work of Waldspurger and can be viewed as twisted Siegel-Weil formulas.
Results of such type were first obtained around forty years ago by Schulze-Pillot \cite{SP84,SP84s} in a more classical language.
Inspired by Snitz's work, Gan \cite{Gan08} proved a cubic analog for division algebras of degree 3 over number fields under certain unramification assumption. 
In another direction, Du and Yang \cite{DY19} considered the integral of twisted theta functions and called the induced equality between the arithmetic pairing of twisted arithmetic divisors on modular curves with the normalized metric Hodge line bundle and the derivative of Eisenstein series as the twisted arithmetic Siegel-Weil formula.

Let $K_4$ be a non-Galois quartic CM field and $F_2$ its real quadratic subfield.
Denote by $\sigma$ the non-trivial element in $\Gal(F_2/\Q)$.
Then the reflex field of $K_4$ with CM type $\Phi=\{ 1,\sigma \}$ is another non-Galois quartic CM field $\wk_4$ with real quadratic subfield $\wf_2$.
The composition of $\wk_4$ and $K_4$ is a degree 8 number field $M_8$, which is Galois over $\Q$ and $\Gal(M_8/\Q) \cong D_4$ the real dihedral group. 
Extend $\sigma$ to an element in $\Gal(M_8/\Q)$ and by a slightly abuse of notation, we still denote it as $\sigma$.
The subfield of $M_8$ fixed by $\sigma$ is a real quadratic field $F$.

In this paper, we use the theory of Doi-Naganuma lift to study the global theta lift of automorphic characters over $\wk_4$ and obtain twisted Siegel-Weil formulas.
Doi and Naganuma \cite{DN69} gave a construction of holomorphic Hilbert modular forms of parallel weight over real quadratic fields from elliptic modular forms, which is the first instance of base change \cite{Saito75, Shintani79, Langlands80}.
Let $G=\SL_2$.
Doi-Naganuma lift can be realized as theta lifting from $G_\Q$ to $\mathrm{O}(2, 2)$, which is isogenous to $G_{\wf_2}$ \cite{Zagier75, Kudla78}. 
This provides one perspective to understand the diagonal restriction of Hilbert modular forms.

Let $W$ be the quadratic space $\wk_4$ over $\wf_2$ with quadratic form $\Nm_{\wk_4/\wf_2}$.
For any $\phi\in \Sc(W(\A_{\wf_2})) = \Sc (\A_{\wk_4})$ and $\chi$ a Hecke character on $\A_{\wk_4}^\times$, the twiste theta integral against $\chi$ is
\begin{align}\label{theta wk4 wk2}
\theta_{\wk_4,\chi} (\gb,\phi) := \int_{[\SO(W)]} \theta_{\wk_4} (\gb, h,\phi) \chi(h) dh,\quad
\gb\in\SL_2(\A_{\wf_2}),
\end{align}
where $[\cdot]$ is the automorphic quotient.
If $\chi$ is trivial, then the theta integral is a Hilbert Eisenstein series by Siegel-Weil formulas.
In the following, we will consider the quadratic character $\chi$ defined in Lemma \ref{wk4 character}.

Inspired by the results in \cite{LZ25}, where Hilbert Eisenstein series is realized the Doi-Naganuma lift of elliptic Eisenstein series, we consider the Doi-Naganuma lift of Hecke's integral.
Let $U$ the the quadratic space $F$ over $\Q$ with quadratic form $\Nm_{F/\Q}$.
For any Schwartz function $\Xi\in \Sc (U(\A))$, the theta lift of an automorphic character $\rho$ on $F^\times\backslash\A_F^\times$ is
\begin{align}\label{theta F Q}
\theta_{\rho} ( g,\Xi) = \int_{[\SO(U)]} \theta ( g,  h, \Xi) \rho(h) dh,\quad g\in \SL_2(\A),
\end{align}
which is called Hecke's integral.

Let $V=\Q^2\oplus\wf_2$ be a quadratic space with quadratic form $Q(x,y,\nu) := xy-\Nm \nu$.
For any Schwartz function $\varphi\in \Sc (V(\A))$, the Doi-Naganuma lift of Hecke's integral is
\begin{align}\label{doi naganuma hecke}
I(\hb,\varphi,\Xi,\rho) := \int_{[G]} \theta_V ( g ,\hb,\varphi)  \theta_{\rho} ( g,\Xi) dg,\quad \hb\in \SL_2(\A_{\wf_2}).
\end{align}
For the extension of $\theta_{\wk_4,\chi} (\gb,\phi),\theta_{\rho} ( g,\Xi)$ and $I(\hb,\varphi,\Xi,\rho)$ from $\SL_2$ to $\GL_2$, see Section \ref{weil repn GL2} and \ref{theta integrals} for more details.
For certain characters $\chi$ and $\rho$, we realize the twisted theta integral $\theta_{\wk_4,\chi} (\gb,\phi)$ as Doi-Naganuma lift of $\theta_{\rho} ( g,\Xi)$.

\begin{theorem}[Twisted Siegel-Weil formulas]\label{twisted sw formula}
Let $\chi$ be the Hecke character defined in Lemma \ref{wk4 character}, and $\rho$ the Hecke character defined in Lemma \ref{matching central characters} satisfying
\begin{align*}
\chi_{\wk_4/\wf_2} \chi  = (\chi_{\wf_2}\chi_{F}\rho^{-1}) \circ \Nm_{\wf_2/\Q}
\end{align*}
on $\A_{\wf_2}^\times$.
Then, for any Schwartz function $\phi\in\Sc (W(\A_{\wf_2}))$, there exists Schwartz functions $\varphi\in \Sc (V(\A))$ and $\Xi\in \Sc (U(\A))$ such that 
\begin{align}
    \theta_{\wk_4,\chi} (\gb,\phi) = I (\hb,\varphi,\Xi,\rho),\quad \gb=\hb\in \GL_2(\A_{\wf_2}).
\end{align}
\end{theorem}

This result could be interpreted as a base change of Jacquet-Langlands correspondence. 
Assume now $K$ is a quadratic extension of a number field $L$. 
Let $\eta$ be a Hecke character of $K^\times\backslash\A_K^{\times}$. 
For any prime ideal $\mathfrak{P} \subset K$ lying over a prime ideal $\mathfrak{p}\subset L$, let $K_\mathfrak{P}^1$ be the compact subgroup of norm 1 elements of $K_\mathfrak{P}$, identified with $\SO(K_\mathfrak{P},\Nm_{K_\mathfrak{P}/L_\mathfrak{p}})$. 
The subspace
$$S(K_\mathfrak{P}, \eta):=\left\{f \in S(K_\mathfrak{P}): f\left(h^{-1} x\right)=\eta(h) f(x), \forall h \in K_\mathfrak{P}^1, \forall x \in K_\mathfrak{P}\right\}$$
is invariant under the Weil representation of $\mathrm{SL}_2(L_\mathfrak{p})$, and extended to an admissible and irreducible representation $\pi_{\eta,\mathfrak{p}}$ of $\GL_2(L_\fp)$ with central character $\chi_{K_\mathfrak{P} / F_\mathfrak{p}} \eta |_{F_\mathfrak{p}^\times}$ by Jacquet and Langlands \cite[Proposition 1.5]{JL70}.
Globally, for any Schwartz function $\phi\in \Sc (\A_K)$, the theta lift of $\eta$ is
\begin{align*}
\theta_{\eta} (\gb,\phi) := \int_{[\SO(K)]} \theta_{K} (\gb, h,\phi) \eta(h) dh,\quad
\gb\in\SL_2(\A_{L})
\end{align*}
and extended to a function on $\GL_2(\A_{L})$.
We denote by $\Theta(\eta)$ the automorphic representation of $\GL_2(\mathbb{A}_L)$ generated by $\theta_\eta( \gb, \phi)$.
Let $\pi_\eta$ be the restricted tensor product of the local representations $\pi_{\eta,\mathfrak{p}}$.
Then
\begin{align}
    \pi_\eta \cong \Theta(\eta)
\end{align}
as automorphic representations \cite[Section 13]{HK91}. More generally, the Jacquet-Langlands correspondence gives an 1-1 correspondence between irreducible two dimensional representations of the Weil group and supercuspidal representations of $\mathrm{GL}_2(L)$.

For the quadratic extension $\wk_4/\wf_2$ and $F/\Q$, we denote by $\pi_\chi$ and $\pi_\rho$ the Jacquet-Langlands correspondence of $\chi$ and $\eta$, which are isomorphic to the representation generated by twisted theta integrals in the form \eqref{theta wk4 wk2} and \eqref{theta F Q} respectively.
Doi-naganuma lift, as a base change of modular forms, gives rise to the base change of Jacquet-Langlands correspondence as follows.

\begin{theorem}
Let $\chi$ and $\rho$ be characters as in Theorem \ref{twisted sw formula}.
Denote by
\begin{align*}
\Theta(\pi_\rho) := \left\{ I (\hb,\varphi,\Xi,\rho) :  \varphi\in \Sc (V(\A)), \Xi\in \Sc (U(\A)) \right\}
\end{align*}
the Doi-Naganuma lift of $\pi_\rho$. 
Then
\begin{align}
    \pi_\chi \cong \Theta(\pi_\rho).
\end{align}
\end{theorem}

The idea of the proof is to apply an analogue of Kudla's matching principle \cite{Kudla03} to prove local twisted Siegel-Weil formulas.
For non-Galois quartic CM fields, our method doesn't apply to trivial characters since they don't satisfy the condition in Lemma \ref{matching central characters}.
However, for the other two kinds of degree 4 CM fields: biquadratic fields and cyclic fields, the author and Li showed that Hilbert Eisenstein series could be realized as the Doi-naganuma lift of elliptic Eisenstein series \cite{LZ25}.

Siegel-Weil formula is also an important tool to study values at complex multiplication points (CM values) of modular functions \cite{Sch09,BY06}.
Bruinier, Kudla and Yang 
interpreted them as integral of Borcherds forms over torus and applied Siegel-Weil formulas to show they are linear combinations of coefficients of the derivative of Eisenstein series and the input of Borcherds forms \cite{BKY12}.
One crucial fact is that the coefficients of the derivative of incoherent Eisenstein series are rational multiples of logarithm of algebraic integers, see, for example, \cite{HY12,YYY21}.
Similar phenomenon also occurs in twisted Siegel-Weil formulas. 
In Snitz's work, the Fourier coefficients of his twisted theta integrals are rational numbers.
Translating Snitz's work into modular forms setting, in \cite{HIPS25}, Herrero, Imamo$\overline{\mathrm{g}}$lu, Pippich and Schwagenscheidt used it to study special values of Green's functions on hyperbolic 3-space.

To study real dihedral modular forms, Li \cite{Li16} introduced twisted CM cycles over non-Galois quartic CM fields and showed the twisted CM values of Hilbert modular functions are linear combinations of Fourier coefficients of real dihedral modular forms.
Li also proposed a conjecture about the prime factorization of twisted CM values of Hilbert modular functions \cite[Conjecture 1.5]{Li16}.
Motivated by this conjecture, in this paper, we apply twisted Siegel-Weil formulas to study the values at twisted CM cycles (twisted CM values) of Borcherds forms on Hilbert modular surface. 

Let $V_2$ be the quadratic space $\Q^2 \oplus F_2$ with quadratic form $xy-\Nm$.
For an even integral lattice $ L \subset V_2$, we denote by $L^{\vee} \subset V$ its dual lattice, and
$\widehat{L}:=L \otimes \widehat{\mathbb{Z}}, \widehat{L}^{\vee}:=L^{\vee} \otimes \widehat{\mathbb{Z}}$.  
Let $K\subset \Gspin(V_2)(\A_f)$ be the open compact subgroup fixing $\widehat{L}$ and $\Dc$ the Grassmannian of negative oriented two dimensional subspaces of $V_2(\R)$.
Then the connected component of the complex points of the associated orthogonal Shimura varieties is isomorphic to the Hilbert modular surface $X_{\Gamma_K}\cong  \Gamma_K\backslash\h^2$, where $\Gamma_K = K \cap \SL_2(\Oc_{\wf_2})$, see Section 3 for more information.
When $\Gamma_K = \SL_2(\Oc_{F_2})$, we write $X_{F_2} = X_{\Gamma_K}$.
As usual, we let $H_{k, L}(\Gamma)$ denote the space of harmonic Maass forms valued in $\mathbb{C}\left[L^{\vee} / L\right]$ of weight $k \in \frac12\mathbb{Z}$ and representation $\rho_L$ on a congruence subgroup $\Gamma \subset \SL_2(\Z)$, see \cite{BF04} for more details. 
If $\Gamma = \SL_2(\Z)$, we will omit it from the notation.

For a $K$-invariant Schwartz function $\varphi \in \Sc(V_2(\widehat{\mathbb{Q}}))^K$ and $\tau = u+iv \in \mathbb{H}$, the theta function
\begin{equation*}
\theta(\tau, z, h, \varphi)= \sum_{x \in V_2(\mathbb{Q})} \omega \left(g_\tau\right) \varphi_{\infty}(x, z) \varphi\left(h^{-1} x\right),\quad
g_\tau = n(u)m(\sqrt{v}),
\end{equation*}
is an automorphic function of $(z, h) \in \Dc \times \Gspin(V_2)(\A_f)$, where $\varphi_{\infty}(x, z)$ is the Gaussian.
Let $\varphi_\mu$ be the characteristic function of $\mu+\widehat{L}$. The vector valued theta function
$$\Theta_L(\tau, z, h):= \sum_{\mu \in L^{\vee} / L} \theta \left(\tau, z, h, \varphi_\mu \right) \mathfrak{e}_\mu$$
is a non-holomorphic modular form of weight 0 in $\tau$.
For $r\in \N_0$, let $R_{\tau}^r$ be the iterated raising operator. 
We consider the regularized theta lift of $f\in H_{-2r, L}$ 
\begin{align*}
\Phi_f^r(z, h) & =(4 \pi)^{-r} \int_{\mathcal{F}}^{\reg}\left\langle R_\tau^r f(\tau), \overline{\Theta_L(\tau, z, h)}\right\rangle d \mu(\tau),
\end{align*}
where $\mathcal{F}$ is the fundamental domain of $\SL_2(\Z) \backslash \mathbb{H}$.
Automorphic forms arising in this way are called Borcherds forms \cite{Kudla03}.

Let $\mathfrak{d}_{K_4/F_2}$ be the relative different of the extension $K_4/F_2$ and $\widetilde{D} = \Nm_{F_2/\Q} \mathfrak{d}_{K_4/F_2}$.
Consider the $\wf_2$-quadratic spaces 
\begin{align*}
W^+ = (\wk_4, \frac{\Nm_{\wk_4/\wf_2}}{\sqrt{\widetilde{D}}}),\quad 
W^- = (\wk_4,  \frac{\Nm_{\wk_4/\wf_2}}{-\varepsilon\sqrt{\widetilde{D}}}),
\end{align*}
where $\varepsilon\in \wf_2^+$ such that $-\varepsilon$ is locally a norm at all finite places.
For any CM point represented by a pair $(\mathfrak{a},r)$, where $\mathfrak{a}\subset F_2$ is a fractional ideal and $r\in F_2^+$, the map $\iota$ \eqref{iso CM point} gives rise to an isomorphism between $V_2$ and $(\wk_4, \frac{-1}{\Nm \mathfrak{a}} \tr_{\wf_2/\Q} \frac{\Nm_{\wk_4/\wf_2}}{\sqrt{\widetilde{D}}})$.
Let $T^+$ be the preimage of $\Res_{\wf_2/\Q} \SO (W^+)$ in $\Gspin (V_2)$ and $K_{T^+} := T^+(\widehat{\Q}) \cap K$. Then the positive definite subspace $W^+\otimes_{\wf_2,\sigma^2} \R$ gives rise to CM points $z^\pm_+ \in \Dc$ up to orientation, and hence a zero cycle
\begin{equation*}
Z(W^+) := T^+(\mathbb{Q}) \backslash \{z^\pm_+\} \times T^+(\widehat\Q) / K_{T^+}
\end{equation*}
on $X_{F_2}$ defined over $\wf_2$. 
It's Galois conjugate under the action of $\sigma$ is the CM cycle $Z(W^-)$ similarly defined by the quadratic space $W^-$ \cite[Section 2]{BKY12}.
Together, they give us the CM cycle $Z(W) := Z(W^+) + Z(W^-)$ on $X_{F_2}$, which is defined over $\Q$, see also \cite[Section 3]{BY06}. 

Any character $\chi$ on $\SO(W)$ naturally induces a character $\chi_{T^+}$ on $T^+$ via the surjection in \eqref{torus so}.
In this paper, we consider the character $\chi$ as in Theorem \ref{twisted sw formula}.
Using the action of general ideal class group on CM points, we define the twisted CM cycle
\begin{align*}
Z_\chi(W^+)= \sum_{\gamma \in T^+(\mathbb{Q}) \backslash T^+(\widehat\Q) / K_{T^+}} \chi_{T^+}(\gamma) \gamma \bullet z_+^\pm
\end{align*}
and $Z_\chi(W^-)$ similarly. 
In \cite{BLY22}, Bruinier, Li and Yang applied deformed theta integrals and Siegel-Weil formulas to show the CM values of Borcherds forms are described by Whittaker coefficients of Hilbert Eisenstein series.
Calculating the Whittaker functions of the Doi-Naganuma lift of deformed theta integrals in the same spirit, we use twisted Siegel-Weil formulas to show that the twisted CM values are algebraic multiple of logarithm of units, explicitly given by the Whittaker functions of twisted theta integrals in the form \eqref{theta wk4 wk2}.

\begin{theorem}[Twisted CM values] \label{tsw introduction}
For $r \in \mathbb{N}_0$ and an integral lattice $L \subset V_2(\Q)$, let
\begin{align*}
f=\sum_{m \in \mathbb{Q}, \mu \in L^{\vee} / L} c(m, \mu) q^m \mathfrak{e}_\mu \in M_{-2 r, \rho_L}^{!}
\end{align*}
be a weakly holomorphic modular form with rational Fourier coefficients, which has vanishing constant term when $r=0$. 
Let $Z_\chi (W^\pm)$ be the twisted CM cycles defined in Equation \eqref{twisted ZW}.
Then there exists $\kappa \in \Q$ depending on $f$ such that 
\begin{align}
\begin{split}
&\Phi_f^r\left(Z_\chi\left(W^+\right)\right)-(-1)^r \Phi_f^r\left(Z_\chi\left(W^-\right)\right)  
= \frac{\deg Z(W)}{-4} \sum_{m>0, \mu \in L^{\vee} / L} c(-m, \mu) m^r \\
&\quad \times \sum_{\substack{\alpha \in \wf_2^+, \operatorname{Tr}(\alpha)=m \\ \exists \lambda_\alpha\in F, \text{ s.t. } \Nm \alpha = \Nm \lambda_\alpha^{-1} }} P_r\left(\frac{\alpha-\alpha^{\prime}}{m}\right) 
W_f(\alpha,\phi_\mu) \log (|\lambda_\alpha/\lambda_\alpha'|\varepsilon_\rho^\kappa),
\end{split}
\end{align}
where $\phi_\mu$ is the image of $\varphi_\mu$ under the isomorphism $\iota$ \eqref{iso CM point}, and 
\begin{align*}
W_f(\alpha,\phi_\mu)
= \int_{\A_{\wk_4^1,f}}   \phi_\mu \left(\alpha \sqrt{\widetilde{D}} \beta^{-1} h^{-1}\right) \chi (h) d h 
\end{align*}
if there exists $\beta \in \A_{\wk_4,f}$ such that $\sqrt{\widetilde{D}} \alpha = \Nm\beta$ and 0 otherwise. 
\end{theorem}

\begin{remark}
Note that $P_r\left(\frac{\alpha-\alpha^{\prime}}{m}\right) $ is a rational multiple of $\sqrt{D}^r$ and $W_f(\alpha,\phi_\mu)\in \Q$.
\end{remark}

\subsection*{Proof Strategy}
Motivated by the twisted CM cycles in \cite{Li16}, we introduce a degree 32 Galois extension $M_{32}$ of $\Q$ containing $M_8$ such that $\Gal(M_{32}/\Q) \cong (\Z/4\Z)^2 \rtimes \Z/2\Z$ and $\Gal(M_{32}/\wk_4) \cong \Z/8\Z$. 
The restriction of the order 8 character on $\Gal(M_{32}/\wk_4)$ to $\SO(W)$ via the Artin map is a quadratic character $\chi$ by Lemma \ref{wk4 character}.
Matching the central characters of the twisted theta integral \eqref{theta wk4 wk2} and the Doi-Naganuma lift of Hecke's integral \eqref{doi naganuma hecke}, we obtain a Hecke character $\rho$ on $\A_F^\times$ satisfying the condition in Lemma \ref{matching central characters}.

To prove the twisted Siegel-Weil formulas, i.e., the base change of Jacquet-Langlands correspondence being an isomorphism, we utilize the extension of Weil representation to a subgroup of $\GL_2 \times\GO(V)$ to obtain adelic Hilbert modular forms on $\GL_2(\A_{\wf_2})$, see Section \ref{weil repn GL2} and \ref{theta integrals} for more details.
Then we calculate the Whittaker functions of the Whittaker newform $\theta_{\wk_4,\chi}(\gb,\phi_0)$ in the theory of Jacquet-Langlands correspondence in Section \ref{twisted theta integral}. 
Applying the mixed model of Weil representation, we unfold the Doi-Naganuma lift of Hecke's integral to obtain the Whittaker functions of $I(\hb,\varphi_0,\Xi_0,\rho)$ for special $\SL_2(\widehat{\Z})$-invariant Schwartz functions $(\varphi_0,\Xi_0)$.

After analyzing the splitting behavior of unramified primes in $M_{32}$, we show the Whittaker functions of $\theta_{\wk_4,\chi}(\gb,\phi_0)$ and $I(\hb,\varphi_0,\Xi_0,\rho)$ are the same at almost all finite places.
Then the strong multiplicity one theorem for Hilbert modular forms implies that they generate the same cuspidal automorphic representation.
This means the base change of Jacquet-Langlands correspondence from $\Q$ for $\rho$ to $\wf_2$ for $\chi$ is an isomorhpism, and twisted Siegel-Weil formulas follow immediately.
Note that instead of using principal series, an irreducible representation of $\SL_2$, to prove Siegel-Weil formulas, we use Jacquet-Langlands correspondence for automorphic representations of $\GL_2$ to construct twisted Siegel-Weil formulas.

In the spirit of Kudla's integral of Borhcerds forms \cite{Kudla03}, we formulate the twisted integral of Borcherds forms against the character $\chi$. 
Using the idea of \cite[Proposition 4.7]{BLY22}, we calculate the Whittaker functions of the Doi-Naganuma lift of deformed theta integrals and apply Stokes' theorem to show that the twisted integral of Borcherds forms is explicitly described by the Fourier coefficients of twisted theta integrals against $\chi$.
Note that, in \cite{Kudla03}, the integral of Borcherds forms concerns the coefficients of Eisenstein series. However, on $\SL_2(\A_{\wf_2})$, they possess one same property, namely, the Whittaker coefficients lies in $\Q(\phi)$, the field generated by the values of the associated Schwartz function $\phi$. Similarly as in \cite{BKY12}, we express twisted CM values of Borcherds forms as twisted integrals of Borcherds forms, and obtain twisted CM value formulas using the Fourier coefficients of twisted theta integrals under certain assumptions.

\subsection*{Outlook and generalization}

In this paper, we only consider the base change of Jacquet-Langlands correspondence for the special characters $\chi$ and $\rho$ as in Therem \ref{twisted sw formula}. More generally, we expect the base change is an isomorphism for any pair of characters $(\chi,\rho)$ satisfying Lemma \ref{matching central characters}, where $\chi$ determines the value of $\rho$ on a subgroup of $\A_{F}^\times$. 
It would also be interesting to study the behavior of $L$-functions and $\epsilon$ factors under the base change by Doi-Naganuma lift.
More generally, we can consider the twisted CM values of theta lift of harmonic Maass forms and explore the twisted Rankin-Selberg type $L$-functions.
We plan to address this in the near future.

As seen in section 6, the twisted CM values rely on the choice of the isomorphism between the finite adelic points of $W^+$ and $W^-$ unlike the case of CM values \cite[Lemma 4.2]{BKY12} due to the Siegel section map. 
To prove the conjecture in \cite{Li16}, we need to know more about the behavior of twisted CM values with respect to such isomorphisms and the difference between the twisted CM cycles constructed by torus and ideal class groups.
Note that the numerical examples in \cite{Li16} coincide with Theorem \ref{tsw introduction}.

In the work of \cite{LZ25}, Li and the author applied invariant vectors to match the Hilbert Eisenstein series at all places. 
However, at ramified places, the invariant vectors don't work in our setting for twisted Siegel-Weil formulas by Proposition \ref{value T1 ramified}.
It needs some new ideas to explicitly match twisted theta integrals.
\\

This paper is organized as follows. 
In section 2, we introduce the extension of numbers fields related to non-Galois quartic CM fields and construct the characters appearing later in the twisted Siegel-Weil formulas. 
Section 3 contains the Borcherds forms on orthogonal Shimura varieties and defines the twisted CM cycles. 
Section 4 calculates the Whittaker functions of twisted theta integrals and Doi-Naganuma lift of Hecke's integrals.
Section 5 matches the local Whittaker functions and proves the twisted Siegel-Weil formulas. Finally, in section 6, we use twisted Siegel-Weil formulas to study the twisted CM values of Borcherds forms.

\subsection*{Acknowledgement}
The author would like to thank Yingkun Li for introducing this project and many helpful suggestions. 
The author also thanks 
Wee Teck Gan for explaining the (irreducibility of) global theta correspondence, and Jan Bruinier, Tuoping Du, Justin Trias for useful discussions.
The author is supported by the Deutsche Forschungsgemeinschaft (DFG) through the Collaborative Research Centre TRR 326 ``Geometry and Arithmetic of Uniformized Structures'' (project number 444845124).

\section{Preliminaries}

\subsection{Extension of number fields}
Let $F_2 =\Q(\sqrt{D_2})$, $D_2\in \N$, be a real quadratic number field. 
Denote by $\sigma$ the non-trivial element in the Galois group $\Gal (F_2/\Q)$.
Let $\alpha = a+b\sqrt{D_2} \in \Oc_{F_2}$ be a totally negative element and $K_4 = F_2 (\sqrt{\alpha})$ be a totally imaginary quadratic extension of $F_2$.
Depending on $\alpha$, $K_4$, as a quartic CM field, might be biquadratic, cyclic or non-Galois \cite[Lemma 4.1]{YY06}. 

In this paper, we assume that $K_4$ is non-Galois over $\Q$. This means $F_2(\sqrt{\alpha}) \neq F_2 (\sqrt{\sigma(\alpha)})$, and $\Q(\sqrt{\alpha\sigma(\alpha)})$ is a real quadratic number field different from $F_2$.
Let $M_8 = \Q (\sqrt{\alpha},\sqrt{\sigma(\alpha)})$. Then $M_8$ is normal and hence the Galois closure of $K_4$.
Moreover, $M_8=K_4 (\sqrt{\alpha\sigma(\alpha)})$ and $\deg M_8 = 8$.
By a slightly abuse of notation, let $\sigma$ and $\tau$ be the automorphisms that map $\left(\sqrt{\alpha}, \sqrt{\sigma(\alpha)}\right)$ to $\left(\sqrt{\sigma(\alpha)},-\sqrt{\alpha}\right)$ and $\left(\sqrt{\sigma(\alpha)}, \sqrt{\alpha}\right)$ respectively. Then 
\begin{align*}
\Gal (M_8/\Q) \cong \langle \sigma,\tau \mid \sigma^4=\tau^2=1, \tau \sigma=\sigma^3 \tau\rangle.
\end{align*}
Note that $\sigma^2$ is the complex conjugation and $K_4$ is the subfield of $M_8$ fixed by $\{1, \tau \sigma\}$. 
Let $F$ be the fixed subfield of $\langle \sigma \rangle$ in $M_{8}$ with $\Gal (F/\Q)=\langle \tau \rangle$. 
In the following, if it's clear in the context, we will use $'$ to denote the non-trivial element in the Galois group of $\wf_2,F_2,F$ over $\Q$.

The CM type $\Phi := \{ 1, \sigma\}$ of $K_4$ is primitive and the associated reflex field is $\wk_4 := \Q (\sqrt{\alpha}+\sqrt{\sigma(\alpha)})$ with reflex CM type $\widetilde{\Phi} = \{ 1, \sigma^{-1} \}$.
And the totally real subfield of $\wk_4$ is $\wf_2 := \Q(\sqrt{D})$, where $D=\alpha\sigma(\alpha)\in\N$.
The type norm map is
\begin{align*}
    \Nm_{\Phi}: K_4 \rightarrow \wk_4 : x\mapsto \prod_{\phi\in\Phi}\phi(x)
\end{align*}
and induces a homomorphism on class groups
\begin{align*}
    \Nm_{\Phi}: \Cl(K_4) \rightarrow \Cl (\wk_4) : \mathfrak{a} \mapsto \mathfrak{a}^\prime,
\end{align*}
where $\mathfrak{a}^\prime\mathcal{O}_{M_8} = \prod_{\phi\in\Phi}\phi(\mathfrak{a})\mathcal{O}_{M_8}$.
The reflex norm map $\Nm_{\widetilde{\Phi}} : \wk_4 \rightarrow K_4$ is defined similarly.

To consider characters motivated by \cite{Li16}, suppose $M_{32}$ is a degree 4 Galois extension of $M_8$ such that
\begin{align}\label{ext cond}
\Gal (M_{32}/\Q) \cong (\Z/4\Z)^2\rtimes \Z/2\Z \cong \{ (a,b,c) : a,b\in \Z/4\Z,c\in\Z/2\Z \},
\end{align}
where the action of $(0,0,1)$ on $(a,b,c)\in \Gal (M_{32}/\Q)$ is given by
\begin{align*}
(0,0,1)(a,b,c)=(b,a,c+1).
\end{align*}
We choose representatives of $\sigma$ and $\tau$ as $(1,0,0)$ and $(1,0,1)$ respectively. 
As a quotient of $\Gal(M_{32}/\Q)$, 
\begin{align*}
\Gal (M_8/\Q) \cong \langle (1,0,0),(1,0,1) \rangle \cong \{ (a,0,c) : a\in \Z/4\Z,c\in\Z/2\Z \}.
\end{align*}
Let $K_8$ be the fixed subfield of the non-normal subgroup $\langle (1,0,0) \rangle$ in $M_{32}$. 
Then $F\subset K_8$ and $K_8/\Q$ is a non-Galois extension. 
Moreover, $M_{32}$ can be constructed from $K_8$ by the proof of \cite[Proposition 2.1]{Li16}.
Indeed, if $K_8 \cong F[X] /\left(f(X)\right)$, then 
\begin{align*}
M_{32} \cong F[X, Y] /\left(f(X), \tau(f)(Y)\right).
\end{align*}
In summary, we have the following diagram of extension of number fields.

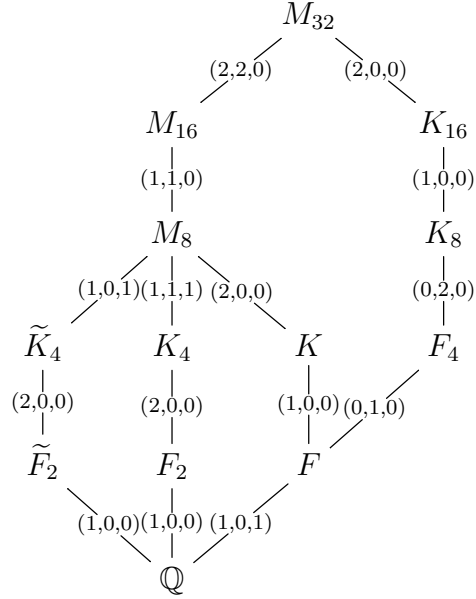
\[\xymatrix{
& & M_{32}\ar@{-}[ld]|-{(2,2,0)} \ar@{-}[dr]|-{(2,0,0)}  &   \\
& M_{16}  \ar@{-}[d]|-{(1,1,0)} &   &    K_{16} \ar@{-}[d]|-{(1,0,0)} \\
 & M_8 \ar@{-}[d]|-{(1,1,1)} \ar@{-}[ld]|-{(1,0,1)} \ar@{-}[rd]|-{(2,0,0)} &   & K_8 \ar@{-}[d]|-{(0,2,0)} \\
\wk_4 \ar@{-}[d]|-{(2,0,0)} & K_4 \ar@{-}[d]|-{(2,0,0)}  & K \ar@{-}[d]|-{(1,0,0)}  & F_4 \ar@{-}[dl]|-{(0,1,0)} \\
\wf_2 \ar@{-}[rd]|-{(1,0,0)} & F_2 \ar@{-}[d]|-{(1,0,0)}  & F \ar@{-}[ld]|-{(1,0,1)} & \\
& \Q  & & \\
}\]

Let $L\subset M_{32}$ be a subfield and $\fp\subset L$ a prime ideal.
For another subfield $N\subset M_{32}$ containing $L$ and a prime ideal $\mfp \subset N$ lying over $\fp$, denote $\bfrob_\mfp^L \in \Gal(N/L)$ as the image of $\mfp$ under the Artin map when $\fp$ is unramified in $N$. 
Furthermore, if $N/L$ is abelian, set $\frob_\fp^N \in \Gal(N/L)$ as the image of $\fp$ under the Artin map.
If $N=M_{32}$ or $L=\Q$, we will omit $N$ or $L$ from the notation $\frob_\fp^N$ or $\bfrob_\mfp^L$ respectively.

The splitting behavior of a prime $p\in\Q$ in $F,\wf_2,\wk_4$ is given as follows.
Suppose $p$ is unramified in $M_8$.
Let $\mfp \subset M_{8}$ be a prime ideal lying over $p$. 

If $\bfrob_\mfp = (a,0,0)$, then $p=\mathfrak{q}_1\mathfrak{q}_2$ is split in $F$, and $\frob_{\mathfrak{q}_1}^{M_8}=(a,0,0),\frob_{\mathfrak{q}_2}^{M_8}=(-a,0,0)$;
\\
I.(1) If $a=0$, then $p=\mathfrak{p}_1\mathfrak{p}_2$ is split in $\wf_2$ and $\frob_{\mathfrak{p}_1}^{M_8}=\frob_{\mathfrak{p}_2}^{M_8}=(0,0,0)$;
$\mathfrak{p}_1=\mathfrak{P}_{11}\mathfrak{P}_{12},\mathfrak{p}_2=\mathfrak{P}_{21}\mathfrak{P}_{22}$ are split in $\wk_4$ and $\frob_{\mathfrak{P}_{ij}}^{M_8}=(0,0,0)$, $i,j=1,2$.
\\
I.(2) If $a=2$, then $p=\mathfrak{p}_1\mathfrak{p}_2$ is split in $\wf_2$ and $\frob_{\mathfrak{p}_1}^{M_8}=\frob_{\mathfrak{p}_2}^{M_8}=(2,0,0)$;
$\mathfrak{p}_1=\mathfrak{P}_{1},\mathfrak{p}_2=\mathfrak{P}_{2}$ are inert in $\wk_4$ and $\frob_{\mathfrak{P}_{i}}^{M_8}=(0,0,0)$, $i=1,2$.
\\
I.(3) If $a=1,3$, then $p=\mathfrak{p}$ is inert in $\wf_2$ and $\frob_{\mathfrak{p}}^{M_8}=(2,0,0)$;
$\mathfrak{p}=\widetilde{\mathfrak{P}}$ is inert in $\wk_4$ and $\frob_{\widetilde{\mathfrak{P}}}^{M_8}=(0,0,0)$.

If $\bfrob_\mfp = (a,0,1)$, then $p=\mathfrak{q}$ is inert in $F$ and $\frob_{\mathfrak{q}}^{M_8}=(0,0,0)$; 
\\
II.(1) If $a = 0,2$, then $p=\mathfrak{p}$ is inert in $\wf_2$ and $\fp = \mfp_1\mfp_2$ is split in $\wk_4$ with $\frob_{\mathfrak{P}_{i}}^{M_8}=(0,0,0)$, $i=1,2$. 
\\
II.(2) If $a=1$, then $p=\mathfrak{p}_1\mathfrak{p}_2$ is split in $\wf_2$ and $\frob_{\mathfrak{p}_1}^{M_8}=(1,0,1),\frob_{\mathfrak{p}_2}^{M_8}=(3,0,1)$.
$\mathfrak{p}_1=\mathfrak{P}_{11}\mathfrak{P}_{12}$ is split in $\wk_4$ and $\frob_{\mathfrak{P}_{11}}^{M_8}=\frob_{\mathfrak{P}_{12}}^{M_8}=(1,0,1)$; 
$\mathfrak{p}_2=\mathfrak{P}_{2}$ is inert in $\wk_4$ and $\frob_{\mathfrak{P}_{2}}^{M_8}=(0,0,0)$.
\\
II.(3) If $a = 3$, then $p=\mathfrak{p}_1\mathfrak{p}_2$ is split in $\wf_2$ and $\frob_{\mathfrak{p}_1}^{M_8}=(3,0,1),\frob_{\mathfrak{p}_2}^{M_8}=(1,0,1)$.
$\mathfrak{p}_2=\mathfrak{P}_{21}\mathfrak{P}_{22}$ is split in $\wk_4$ and $\frob_{\mathfrak{P}_{21}}^{M_8}=\frob_{\mathfrak{P}_{22}}^{M_8}=(1,0,1)$; 
$\mathfrak{p}_1=\mathfrak{P}_{1}$ is inert in $\wk_4$ and $\frob_{\mathfrak{P}_{1}}^{M_8}=(0,0,0)$.

\subsection{Characters}
To construct twisted Siegel-Weil formulas, we consider characters as in \cite{Li16}.
Let the character 
\begin{align*}
\psi : \Gal(M_{32}/\wk_4) \cong \langle \tau \rangle \rightarrow \C^\times : \tau \mapsto \zeta_8=\exp (\frac{\pi}{4}i). 
\end{align*}
Then $\Ind_{\wk_4}^{\wf_2} \psi \cong \Res^{\Q}_{\wf_2}\rho_\varphi $, see also \cite[Equation (4.3)]{Li16}, where the character
\begin{align*}
\varphi: \operatorname{Gal}\left(M_{32} / F\right) & \longrightarrow \mathbb{C}^{\times},\ (a, b, 0)  \mapsto e^{\pi i b / 2},
\end{align*}
and $\rho_\varphi := \Ind_F^\Q \varphi$ is the induced irreducible representation of $\Gal (M_{32}/\Q)$. 
The restriction of $\rho_\varphi$ to $\Gal (M_{32}/\wk_4)$ is reducible and can be decomposed as
\begin{align*}
\rho_\varphi \mid_{\Gal (M_{32}/\wk_4)} \cong \psi \oplus \widetilde{\psi},
\end{align*}
where $\widetilde{\psi}$ is a character on $\Gal (M_{32}/\wk_4)$ satisfying $\widetilde{\psi}(\tau)=\psi (\tau^5)$. 
The following lemma will be used to define twisted theta integrals and twisted CM cycles later.
For any subfield $L\subset M_{32}$, let $\Art_{L} : \A^\times_{L}\rightarrow \Gal (M_{32}/L)$ be the Artin map.

\begin{lemma}\label{wk4 character}
Denote $\wk_4^1$ by the kernel of $\Nm_{\wk_4/\wf_2}$ and $\A_{\wk_4^1} := \A_\Q \otimes \wk_4^1$.
Then the composition 
\begin{align*}
\chi: \A_{\wk_4^1} \overset{\Art_{\wk_4}}{\longrightarrow} \Gal(M_{32}/\wk_4) \overset{\psi}{\longrightarrow} \C^\times
\end{align*}
is a quadratic character.
\end{lemma}
\begin{proof}
Let $\mathfrak{p} \subset\wf_2$ be a prime ideal.
If $\mathfrak{p}$ is inert or ramified in $\wk_4$, the image $\Art_{\wk_4}(\A_{\wk_4^1,\mathfrak{p}})$ in $\Gal(M_{32}/\wk_4)$ is trirvial.
When $\mathfrak{p}\Oc_{\wk_4} = \mathfrak{P}_1\mathfrak{P}_2$ is split in $\wk_4$, we denote by $\varpi$ the uniformizer of $\wf_{2,\mathfrak{p}}$. 
Then every element in $\A_{\wk_4^1,\mathfrak{p}}$ can be written as $(\varpi ^nu,\varpi^{-n} u^{-1})$, $u\in \Oc_{\wf_2,\mathfrak{p}}^\times,n\in\Z$. Suppose the image of $\mathfrak{P}_1$ under the Artin map is $\frob_{\mathfrak{P}_1} = \tau^r$, $r\in\Z/8\Z$. Then
\begin{align*}
\Art_{\wk_4}(\varpi ^nu,\varpi^{-n} u^{-1})=\frob_{\mathfrak{P}_1}^n \frob_{\mathfrak{P}_2}^{-n} = \tau^{rn} (\sigma^2\tau^{rn} \sigma^2)^{-1} = \tau^{4rn}.
\end{align*}
Since the Artin map on $\A_{\wk_4}^\times$ is surjective, we have $\Art_{\wk_4}(\A_{\wk_4^1}) = \langle \tau^4 \rangle \cong 4\Z/8\Z$ as a subgroup of $\Gal(M_{32}/\wk_4)$. Hence $\chi$ is a quadratic character.
\end{proof}

From the proof, it's easy to see that the image of $\A_{\wk_4^1} \overset{\Art_{\wk_4}}{\longrightarrow} \Gal(M_{32}/\wk_4)$ is contained in $\Gal (M_{32}/F)$. This enables us to find a character $\rho$ on $\A_F^\times$ satisfying the following conditions, which will be used to establish twisted Siegel-Weil formulas.
Let $\chi_{\wk_4/\wf_2}, \chi_{\wf_2}=\chi_{\wf_2/\Q}$ and $\chi_F=\chi_{F/\Q}$ be the Hecke characters associated to the quadratic extensions $\wk_4/\wf_2, \wf_2/\Q$ and $F/\Q$ respectively.

\begin{lemma}\label{matching central characters}
Define the character $\rho : \A_F^\times \overset{\Art_{F}}{\longrightarrow} \Gal(M_{32}/F) \overset{\eta}{\longrightarrow} \C^\times$, where
\begin{align*}
\eta (1,0,0)=-1, \quad \eta(0,1,0)= i. 
\end{align*}
Then, for $\alpha\in \A_{\wf_2}^\times$ and $a=\Nm_{\wf_2/\Q} \alpha$, we have
\begin{align*}
(\chi_{\wk_4/\wf_2} \psi)(\alpha)
&=(\chi_{\wf_2/\Q}\chi_{F/\Q}\rho^{-1})(a) .
\end{align*}
\end{lemma}
\begin{proof}
Since both sides are Hecke characters on $\A_{\wf_2}^\times$, it suffices to prove the claim for all prime ideals $\fp \subset \wf_2$ lying over rational primes $p$ unramified in $M_{32}$.
We prove the equality when $p = \fp_1\fp_2$ is split in $\wf_2$, and the inert case can be considered similarly. 

Let $\mfp_1 \subset M_{32}$ be any prime ideal lying over the prime ideal $\fp_1$.
Denote by $\bfrob_{\mfp_1}^{\wf_2} = (1,0,0)^a(1,0,1)^b\in\Gal (M_{32}/\wf_2)$ the image of $\mfp_1$ under the Artin map, $a\in\Z/4\Z,b\in\Z/8\Z$.
Then $\mfp_2 := \sigma(\mfp_1)$ is a prime ideal lying over $\fp_2$ and $\bfrob_{\mfp_2}^{\wf_2} = (1,0,0)^a(1,0,1)^b$ if $b$ is even and $\bfrob_{\mfp_2}^{\wf_2} = (1,0,0)^{a+2} (1,0,1)^{5b-6}$ otherwise.

Case $2\mid b$, $4\mid a$: In this case, $p\Oc_{F}=\fq_1\fq_2$ is split in $F$, and $\fp_i = \mfp_{i1} \mfp_{i2}$ is also split in $\wk_4$, $i=1,2$. Moreover, $\frob_{\fq_1} = \frob_{\fq_2} = (1,0,1)^{b}$, and $\frob_{\mfp_{i1}} = (1,0,1)^b, \frob_{\mfp_{i2}} = (1,0,1)^{5b}$. Then
\begin{align*}
&(\chi_{\wk_4/\wf_2} \psi)(\alpha) = \psi(\alpha) = \psi ( (1,0,1)^b (1,0,1)^{5b})^{v_{\fp_1}(\alpha) + v_{\fp_2}(\alpha)} 
= i^{b(v_{\fp_1}(\alpha) + v_{\fp_2}(\alpha))}
\end{align*}
and
\begin{align*}
& (\chi_{\wf_2/\Q}\chi_{F/\Q}\rho^{-1})(a) =  \rho^{-1}(a) = \rho^{-1} ( (1,0,1)^b (1,0,1)^{b})^{(v_{\fp_1}(\alpha) + v_{\fp_2}(\alpha))} = i^{b(v_{\fp_1}(\alpha) + v_{\fp_2}(\alpha))}
\end{align*}
are the same.

Case $2\mid b$, $4\nmid a,2\mid a$: In this case, $p\Oc_{F}=\fq_1\fq_2$ is split $F$, and $\fp_i = \widetilde{\mfp}_i$ is inert in $\wk_4$, $i=1,2$. Moreover, $\frob_{\fq_1} = \frob_{\fq_2} = (1,0,0)^a (1,0,1)^b$, and $\frob_{\widetilde{\mfp}_{i}} = (b,b,0)$. Then
\begin{align*}
&(\chi_{\wk_4/\wf_2} \psi)(\alpha) 
=(-1)^{(v_{\fp_1}(\alpha) + v_{\fp_2}(\alpha))} i^{b(v_{\fp_1}(\alpha) + v_{\fp_2}(\alpha))}
= (\chi_{\wf_2/\Q}\chi_{F/\Q}\rho^{-1})(a) .
\end{align*}

Case $2\nmid b$, $4\mid a$: In this case, $p\Oc_{F}=\fq$ is inert in $F$ with $\frob_\fq = (1,0,1)^{2b}$. 
$\fp_1 = \mfp_{11}\mfp_{12}$ is split in $\wk_4$ and $\frob_{\mfp_{11}} = (1,0,1)^b, \frob_{\mfp_{12}} = (1,0,1)^{5b}$.
$\fp_2 = \widetilde{\mfp}_2$ is inert in $\wk_4$ and $\frob_{\widetilde{\mfp}_2} = (5b-4,5b-4,0)$.
Then
\begin{align*}
&(\chi_{\wk_4/\wf_2} \psi)(\alpha) 
=(-1)^{v_{\fp_2}(\alpha)} \psi ((1,0,1)^{6b})^{v_{\fp_1}(\alpha)} \psi((5b-4,5b-4,0))^{v_{\fp_2}(\alpha)} 
=(-i)^{b(v_{\fp_1}(\alpha) + v_{\fp_2}(\alpha))}
\end{align*}
and
\begin{align*}
& (\chi_{\wf_2/\Q}\chi_{F/\Q}\rho^{-1})(a)
=(-1)^{v_{\fp_1}(\alpha) + v_{\fp_2}(\alpha)} \rho^{-1} ((b,b,0))^{(v_{\fp_1}(\alpha) + v_{\fp_2}(\alpha))} 
=(-i)^{b(v_{\fp_1}(\alpha) + v_{\fp_2}(\alpha))}
\end{align*}
are equal.

Case $2\nmid b$, $4\nmid a,2\mid a$: In this case, $p\Oc_{F}=\fq$ is inert in $F$ with $\frob_\fq = (1,0,1)^{2a+2b}$. 
$\fp_1 = \widetilde{\mfp}_1$ is inert in $\wk_4$ and $\frob_{\widetilde{\mfp}_2} = (a+b,a+b,0)$.
$\fp_2 = \mfp_{21}\mfp_{22}$ is split in $\wk_4$ and $\frob_{\mfp_{21}} = (1,0,1)^{5b-6}, \frob_{\mfp_{22}} = (1,0,1)^{25b-30}$.
Then
\begin{align*}
&(\chi_{\wk_4/\wf_2} \psi)(\alpha) 
= i^{b(v_{\fp_1}(\alpha) + v_{\fp_2}(\alpha))}
= (\chi_{\wf_2/\Q}\chi_{F/\Q}\rho^{-1})(a)
.
\end{align*}
This finishes the proof.
\end{proof}

At the end of this section, we give a criterion when the restriction of a character on norm one elements is trivial.

\begin{lemma}
Let $L/M/N$ be an extension of number fields and $[M:N]=2$. If 
\begin{equation*}
\chi : M^\times\backslash \A_M^\times \overset{\Art_M}{\longrightarrow} \Gal(L/M) \overset{\eta\mid_M}{\longrightarrow} \C^\times
\end{equation*}
is the restriction of a character $\eta : \Gal(L/N) \rightarrow \C^\times$, then $\chi$ is trivial on $\A_{M^1}$. 
\end{lemma}
\begin{proof}
It suffices to consider primes $\fp \subset N$ that are split in $M$.
Suppose $\fp \Oc_{M}=\mfp_1\mfp_2$. Then $\A_{M,\fp}^1\cong N_\fp^\ast : (h,h^{-1})\mapsto h$. 
Let $\varpi$ be a uniformizer of $N_{\mathfrak{p}}$ and write $h=\varpi^n u$, $n\in\Z, u\in \Oc_{N_\fp}^\times$. 
Denote by $\beta$ the non-trivial element in $\Gal(M/N)$, we obtain
$$\chi (h,h^{-1})=\chi (\frob_{\mfp_1}^n \frob_{\mfp_2}^{-n})=\eta (\frob_{\mfp_1}^n \beta \frob_{\mfp_1}^{-n} \beta^{-1} )=1.$$
Therefore, $\chi\mid_{\A_{M^1}}$ is trivial.
\end{proof}

\subsection{Weil representation  for $\GL_2$}\label{weil repn GL2}
For any number field $L$, we denote the ring of integers $\mathcal{O}_{L}$, different $\mathfrak{d}_L = \mathfrak{d}_{L/\Q}$, discriminant $d_L$ and narrow class number $h^+$. 
Let $(V_0,Q_0)$ be a quadratic space over $L$ of even dimension $2n$, $n\in \Z$.
Denote $(\cdot,\cdot)$ as the bilinear form associated to $Q_0$.
Let the character $\psi_L = \psi \circ \tr_{L/\Q}  $, where $\psi$ is the standard additive character on $\A=\A_\Q$ with infinite part given by $\psi_\infty(\cdot) = \e(\cdot) := \exp (2\pi i \cdot)$.
Let $G=\SL_2$.
For any ring $R$, we define the group
\begin{align*}
M(R) &:= \left\{ m(a) = \pmat{a}{}{}{a^{-1}} : a\in R^\times \right\},\quad
N(R) := \left\{ n(b) = \pmat{1}{b}{}{1} : b\in R \right\},\\
T_1 &:=\left\{ d(\alpha) = \pmat{1}{}{}{\alpha} : \alpha \in R^\times  \right\},\quad 
T^1 :=\left\{ t(\alpha) = \pmat{\alpha}{}{}{1} : \alpha \in R^\times  \right\}.
\end{align*}
On the space of Schwartz functions $\Sc(V_0(\A_L))$, there is the Schrödinger model of the Weil representation $\omega = \omega_{\psi_L}$ on $G(\A_L)$.
For $\varphi\in\Sc(V_0(\A_L))$, the action of $\omega$ is given by
\begin{align*}
(\omega(m(a)) \varphi)(x) & =|a|_{\A_L}^{n} \chi_{V_0}(a) \varphi(a x),\quad (\omega(n(b)) \varphi)(x)=\psi_L(b Q_0(x)) \varphi(x), \\
(\omega(w) \varphi)(x) & =\int_{V_0(\mathbb{A}_L)} \varphi(y) \psi((x, y)) d y,\quad (\omega(h) \varphi)(x)=\varphi\left(h^{-1} x\right),
\end{align*}
where $\chi_{V_0} (\cdot) = (\cdot,\det V_0)$, $m(a)\in M(\A_L)$, $n(b)\in N(\A_L)$, $w=\pmat{}{1}{-1}{}$ and $h\in \Spin(V_0)(\A_L)$. 

To construct twisted Siegel-Weil formulas, we need to consider automorphic representations of $\mathrm{GL}_2(\mathbb{A}_{L})$. 
Therefore, we extend the Weil representation to a subgroup of $\operatorname{GO}\left(V_0\right)(\mathbb{A}_L) \times \mathrm{GL}_2(\mathbb{A}_L)$ as in \cite{HK92,Sh72}, where
\begin{align*}
\GO(V_0):=\left\{\sigma \in\GL(V_0): \exists\ \nu(\sigma) \in L^* \text { such that } Q_0(\sigma v)=\nu(\sigma) Q_0(v), \forall\ v \in V_0\right\} .
\end{align*}
The scalar $\nu(\sigma)$ is called the similitude factor of $\sigma \in \GO(V_0)$.
Note that a scalar multiple of the quadratic form doesn't change the similitude factor.
We extend the action of $\mathrm{O}\left(V_0\right)$ to $\GO\left(V_0\right)$ by
\begin{align*}
L(h) \varphi(x)=|\nu(h)|^{-n / 2} \varphi\left(h^{-1} x\right), \quad  h \in \GO\left(V_0\right)(\A_L), \varphi \in \Sc\left(V_0(\A_L)\right).
\end{align*}
This action does not commute with the Weil representation on $\SL_2(\mathbb{A}_L)$, but satisfies the global version of Lemma 1.4 in \cite{JL70}, which we recall now, see also  \cite{popa06}.
Let $h \in \GO\left(V_0\right)(\A_L)$ and $a=\nu(h) \in \A_L^{\times}$. Then
\begin{equation}\label{commute law}
L(h) \omega(g) L(h)^{-1}=\omega\left(\left(\begin{array}{ll}
1 & 0 \\
0 & a
\end{array}\right) g\left(\begin{array}{cc}
1 & 0 \\
0 & a^{-1}
\end{array}\right)\right),
\end{equation}
for all $g \in \SL_2(\mathbb{A}_L)$.
This allows us to extend the Weil representation $\omega$ to the algebraic group
$$R:=\{(g,h) \in \GL_2 \times \GO(V_0) : \nu(h)=\det g\}$$
by defining
\begin{align}\label{weil left}
\omega (g,h) \varphi(x)= \omega\left(g_1\right) L(h) \varphi(x),\quad \text { for }(g,h) \in R(\mathbb{A}_L),
\end{align}
where $g_1= g \pmat{1}{}{}{\det  g^{-1}} \in \mathrm{SL}_2(\mathbb{A}_L)$.
By \eqref{commute law}, there is another equivalent definition:
\begin{align}\label{weil right}
\omega (g,h) \varphi(x)=L(h) \omega \left(g_2\right) \varphi(x),\quad g_2 = \pmat{1}{}{}{\det  g^{-1}}g.
\end{align}
For a Schwartz function $\varphi\in \Sc (V_0(\A_L))$, the associated theta function
\begin{align}\label{theta function}
\theta_{V_0} ( g d(\nu(h)), h,\varphi) := \sum_{x\in V_0(L)} \lb \omega( g d(\nu(h)), h) \varphi \rb (x)
\end{align}
for $(g, h) \in (\SL_2 \times \GO(V_0))(\mathbb{A}_L)$.

\subsubsection{Mixed model}
Using different choices of polarization, there is also the mixed model of Weil representation.
In this paper, we restrict to the rational quadratic space 
\begin{align*}
V=\left\{ X= \pmat{x}{\nu}{\sigma(\nu)}{y} : x,y\in\Q, \nu\in\wf_2 \right\}
\end{align*}
with quadratic form $Q(X) := a_1 \det X$ for $a_1\in\Q$, and refer to \cite{LM80,Kud16} for the general theory.

The group $\GO(V) \cong \Res_{\wf_2/\Q} \GL_2 \times \G_m$, and $(h,a)\in \GO(V)$ acts on $X \in V$ by $a^{-1} h X h^{t,\prime}$. 
Under this identification, the similitude factor is
\begin{align*}
\nu : \GO(V) \rightarrow \G_m : (h,a) \mapsto \nu(h,a) = \Nm \lb \frac{\det h}{a} \rb,
\end{align*}
and induces the subgroup
\begin{align*}
R(V) := \{ (g,(h,a))\in \Res_{\wf_2/\Q}\GL_2 \times \GO(V) : \det (g) = \nu (h,a) \}.
\end{align*}
Define the embedding 
\begin{align*}
j : \Gspin (V) \rightarrow \GO(V) : h \mapsto j(h) := (h,\det h), 
\end{align*}
which is consistent with the action of $\Gspin(V)$. 
We identify $\Res_{\wf_2/\Q}\GL_2$ as a subgroup of $\GO(V)$ by
\begin{align}\label{embedding GL2}
i : \Res_{\wf_2/\Q}\GL_2 \rightarrow \GO(V) : h \mapsto i(h) := (h,\Nm \det h).
\end{align}
Note that $i$ and $j$ are not the same on $\Res_{\wf_2/\Q}\GL_2 \cap \Gspin (V)$ but consistent on $\Res_{\wf_2/\Q}\SL_2$.

We can recover the mixed model of Weil representation from the Schrödinger model by partial Fourier transform as follows, see \cite{LV80} for more information.
We identity $(V,a_1\det) \cong (\Q^2 \oplus \wf_2, a_1xy-a_1\Nm\nu )$ by mapping $ X $ to $(x,y,\nu)$.
Take isotropic lines $\ell^+ = \Q (1,0,0)$ and $\ell^- = \Q (0,1,0)$.
Let $P$ be the parabolic subgroup of $\SO(V)$ fixing the isotropic lines $\ell^\pm$. 
The Levi decomposition of $P$ is $N \rtimes M$ and $N$ is the unipotent radical.
The partial Fourier transform
\begin{align*}
\Fc : \Sc (V(\A)) &\rightarrow \Sc \lb \lb \ell^{-,2} \oplus \wf_2 \rb (\A)\rb,\ \varphi \longmapsto \Fc(\varphi),\\
\Fc(\varphi)\left(\eta_1,\eta_2,\nu\right) &=\int_{\A} \varphi \left( x,\eta_1,\nu \right) \psi (-x\eta_2) dx,\quad \eta_1,\eta_2\in\Q,\nu\in\wf_2,
\end{align*}
gives an isomorphism between the Schrödinger model and the mixed model of Weil representation.
It's easy to verify the following lemma and we omit the proof.

\begin{lemma}\label{weil mixed model}
Let $U_0$ be the rational quadratic space $\wf_2$ with quadratic form $-a_1 \Nm_{\wf_2/\Q}$.
\begin{enumerate}
\item For all~$g\in G(\A_{\wf_2})$ we have
\[\Fc \lb\omega(g)\varphi\rb(\eta_1,\eta_2,\nu) = \omega_0(g)\Fc(\varphi)((\eta_1,\eta_2)g,\nu),\]
where~$\omega_0$ is the Weil representation for~$U_0$.
\item For the action of elements in $P$, we have
\begin{align*}
\Fc\lb \omega ( n(b))\varphi\rb (\eta_1,\eta_2,\nu) &= \e\lb [ (b,\nu)+Q(b)\eta_1]\eta_2\rb\Fc(\varphi)(\eta_1,\eta_2,\nu+b\eta_1),\ b\in \A_{\wf_2},\\
\Fc\lb \omega ( m(\beta))\varphi\rb (\eta_1,\eta_2,\nu) &=  |\beta\beta^{\prime}|^{-1} \Fc\lb\varphi\rb (\beta^{-1}\beta^{\prime,-1}\eta_1, \beta^{-1}\beta^{\prime,-1}\eta_2,\beta\beta^{\prime,-1} \nu),\ \beta \in \A_{\wf_2}^\times.
\end{align*}
\end{enumerate}
\end{lemma}

\begin{remark}
Applying the partial Fourier transform $\mathcal{F}$, we can rewrite the theta function on $V$ as
\begin{align*}
&\theta_{V} ( g d(\nu(\hb)),\hb,\varphi) 
= \sum_{\nu \in \wf_2,\eta\in \Q^2} \Fc (\omega ( g d(\nu(\hb)),\hb) \varphi)(\eta ,\nu),\ (g,\hb) \in (\SL_2,\GO(V))(\A).
\end{align*}
\end{remark}

We recall the invariant vectors constructed in \cite{LZ25} and calculate their patial Fourier transform.
For odd prime $p\mid d_{\wf_2}$ and $p\nmid d_{F}$, let $\varpi_p$ be the uniformizer of $\wf_{2,p}$ and $a_{1,p}=p^{-1}$.
Let $L_p=(p\Z_p)^2\oplus \varpi_p\mathcal{O}_p$. Then 
$L_p^{\vee} / L_p = ( \Z_p^2\oplus \mathcal{O}_p) / ((p\Z_p)^2\oplus \varpi_p\mathcal{O}_p) \cong \operatorname{Sym}_2(\mathbb{Z} / p \mathbb{Z})$ 
and $g\in\GL_2(\Z/p\Z)$ acts on $X\in \operatorname{Sym}_2(\mathbb{Z} / p \mathbb{Z})$ by $(\det g) ^{-1} gXg^t$. 
The set of isotropic vectors in $L_p^{\vee} / L_p$ are given by
\begin{align*}
\left\{ d(1,0,0): d \in(\mathbb{Z} / p \mathbb{Z})^{\times}\right\} \cup\left\{ d\left(j^2, 1, j\right): j \in \mathbb{Z} / p \mathbb{Z}, d \in(\mathbb{Z} / p \mathbb{Z})^{\times}\right\}.
\end{align*}
For the unique quadratic character $\chi_p$ of $(\mathbb{Z} / p \mathbb{Z})^{\times}$, let $\mathfrak{g}_{\chi_p}:=\sum_{j \in(\mathbb{Z} / p \mathbb{Z})^\times} \chi(j) \mathbf{e}(j / p)$ be its Gauss sum. 
Define  $\fu_{3,p} \in \mathcal{S}\left(V\right)(\Q_p)$ by 
\begin{align*}
\fu_{3,p} &= \frac{1}{\mathfrak{g}_{\chi_{p}}(p-1)} 
\sum_{h \in \mathrm{GL}_2\left(\mathcal{O}_p / \varpi_p\Oc_p\right)} \chi_{p}(\operatorname{det}(h)) 
\operatorname{Char}\left(h \cdot\left(\left(1, 0,0\right)+ L_p \right)\right)\\
&=\frac{p}{\mathfrak{g}_{\chi_{p}}}\sum_{d \in(\mathbb{Z} / p \mathbb{Z})^{\times}} \chi_{p}(d)
\left(\operatorname{Char}\left(d(1,0,0)+L_p\right)
+\sum_{j \in \mathbb{Z} / p \mathbb{Z}} 
\operatorname{Char}\left(d\left(j^2, 1, j\right)+L_p\right)\right),
\end{align*}
which is $\SL_2(\Z_p)$-invariant under the Weil representation \cite[Section 5]{LZ25}.
Similarly as \cite[Proposition 4.2]{HL25}, its partial Fourier transform is 
\begin{align*}
&\mathcal{F}\left(\fu_{3,p}\right)(\eta, \nu)
= \chi_{p}(\eta_2) 
\operatorname{Char}\left( p\mathbb{Z}_p \oplus \mathbb{Z}_p^\times \oplus \varpi_p\mathcal{O}_p\right)(\eta, \nu) \\
& + \frac{1}{\mathfrak{g}_{\chi_{p}}}\sum_{d\in(\Z / p \Z)^{\times}} \sum_{j \in (\Z / p \Z)^\times} \chi_{p}(d) \mathbf{e}\left(\frac{d j^2 \eta_2}{p}\right) \operatorname{Char}\left(\left(d+ p\mathbb{Z}_p\right) \oplus \mathbb{Z}_p \oplus\left( d j +\varpi_p\mathcal{O}_p\right)\right)(\eta, \nu) .
\end{align*}
When $\eta=\left(0, \eta_2\right)$, we have
\begin{align}\label{invariant vector 3}
\mathcal{F}\left(\fu_{3,p}\right)(\eta, \nu)= 
\chi_{p}\left(\eta_2\right) \operatorname{Char}\left(\mathbb{Z}_p^{\times} \oplus \varpi_p\mathcal{O}_p\right)\left(\eta_2, \nu\right) .
\end{align}
For odd prime $p\mid (d_F,d_{\wf_2})$, let $\varpi_p$ be the uniformizer of $F_p=\wf_{2,p}$.
Define 
\begin{align}\label{invariant vector 2}
\fu_{2,p} := \sum_{j\in \Z/p\Z} \cha( \varpi_p^{-1}j + \Oc_{\wf_{2,p}} ) \otimes \cha( \varpi_p^{-1}j + \Oc_{F_{p}} ),
\end{align}
which is also $\SL_2(\Z_p)$-invariant \cite[Equation (5.1)]{LZ25}.

\subsubsection{Galois action}
To consider the algebraicity of twisted CM values, we introduce the $\Q$-linear extension of Weil representations in \cite{McGraw03}, see also \cite[Section 2]{BLY22} for lattices with even rank. 

Let $\mathcal{S} (\widehat{V} ; \mathbb{Q}^{\mathrm{ab}} )$ be the subspace of $\mathcal{S} (\widehat{V} ) := \Sc (V(\A_f))$ with values in $\mathbb{Q}^{\mathrm{ab}}$. Using the identification $\mathrm{GL}_2 \cong \mathrm{SL}_2 \rtimes T^1$, the Weil representation $\omega_f$ on $\SL_2(\A_f)$ extends to a $\mathbb{Q}$-linear action of $\mathrm{GL}_2(\widehat{\mathbb{Q}})$ on $\mathcal{S} (\widehat{V} ; \mathbb{Q}^{\mathrm{ab}} )$ via 
\begin{align}
\left(\widetilde{\omega}_f(g, t(a)) \phi\right)(x):=\left(\omega_f(g) \sigma_a(\phi)\right)(x)=\sigma_a\left(\left(\omega_f\left(t(a)^{-1} g t(a)\right)(\phi)\right)(x)\right)
\end{align}
for $g \in \mathrm{SL}_2(\widehat{\mathbb{Q}}), a \in \widehat{\mathbb{Q}}^\times, \phi \in \mathcal{S} (\widehat{V} ; \mathbb{Q}^{\mathrm{ab}} )$, where $\sigma_a \in \operatorname{Gal}\left(\mathbb{Q}^{\mathrm{ab}} / \mathbb{Q}\right)$ 
satisfies $\sigma_a\left(\psi_f\left(a^{\prime}\right)\right)=\psi_f\left(a a^{\prime}\right)$ for all $a^{\prime} \in \widehat{\mathbb{Q}}^{\times}$.
It's easy to verify that the Schwartz functions $\fu_{2,p}$ and $\fu_{3,p}$ are $\SL_2(\widehat{\Z}) \rtimes T^1(\widehat{\Z}) $-invariant. 
Moreover, if $\varphi\in\Sc(V(\A_f))$ is $\Q$-valued and $\SL_2(\widehat{\Z})$-invariant, then $\varphi$ is $\SL_2(\widehat{\Z}) \rtimes T^1(\widehat{\Z}) $-invariant.

\begin{definition}
Let $T^{1,\Delta}$ be the image of $T^{1} \rightarrow \GL_2 \times \Gspin(V)$ under the diagonal embedding.
Consider $F$ as a rational quadratic space $U$ with quadratic form $\Nm_{F/\Q}$.
We say that $(\varphi,\Xi)\in \Sc(V(\A_f))\times \Sc(U(\A_f))$ is $(T^{1,\Delta}(\widehat{\Z}),\widetilde{\omega}_f)$-invariant if 
\begin{align*}
\widetilde{\omega}_f(t)(L(j(t))\varphi,\Xi) = (\varphi,\Xi),\quad t\in T^1(\widehat{\Z}).
\end{align*}
\end{definition}

\subsection{Hilbert modular forms}
We introduce the theory of (adelic) Hilbert modular forms, and refer to \cite{Shimura78,Gar90} for more details.
Let $E$ be a totally real number field of degree $m$ with distinct real embeddings $\sigma_i$, $i=1,\dots,m$.
For an element $a\in E$, set $a_i=\sigma_i(a)$, the $i$-th component of $a$ under the embedding $E \subset E_\mathbb{R} = E\otimes_\mathbb{Q}\mathbb{R}$. 
For a set $S\subset E$, set $S^+$ to be the subset of totally positive elements of $S$.
In particular, $\mathcal{O}_E^{\times, +}$ is the set of totally positive units of $\mathcal{O}$. 
We denote the subset of non-zero elements as $S^*=S\backslash\{ 0 \}$.

Let $\mathfrak{c}$ be a fractional ideal, $\mathfrak{n}$ a nonzero integral ideal of $E$, and $\Gamma_0(\mathfrak{c}, \mathfrak{n})$ the congruence subgroup defined by
\[\Gamma_0(\mathfrak{c}, \mathfrak{n})=\left\{\begin{pmatrix}
a&b\\c&d
\end{pmatrix}: a,d\in\mathcal{O}_E, b\in (\mathfrak{cd})^{-1}, c\in\mathfrak{cnd}_E, ad-bc\in\mathcal{O}_E^{\times, +}\right\}.\]
We briefly denote $\Gamma=\Gamma_0(\mathfrak{c}, \mathfrak{n})$ in this section.
It is well-known that $\Gamma$ acts discontinuously on $\mathbb{H}^{m}$  with finite covolume via the componentwise M\"obius transformation. 
For $\mathbf{k}=(k_1,\dots,k_d)\in \Z^m$, we denote by $M_{\mathbf{k}}\left(\Gamma ,\delta \right)$ the space of Hilbert modular forms of weight $\mathbf{k}$, level $\Gamma$ with character $\delta : (\Oc_E/\mathfrak{n})^\times \rightarrow \C^\times$, i.e. the space of holomorphic functions $f(z)$ on $\mathbb{H}^{m}$ such that 
\begin{align*}
f(\gamma(z))= \mathrm{det}(\gamma)^{-\frac{\mathbf{k}}{2}}N(c z+d)^{\mathbf{k}} \delta^{-1}(d) f(z)
\end{align*}
for $\gamma=\pmat{a}{b}{c}{d} \in \Gamma$ and  $z=\left(z_{1}, \ldots, z_{m}\right) \in \mathbb{H}^{m}$. 
Here 
\begin{align*}
\mathrm{det}(\gamma)=(\mathrm{det}(\gamma_1),\dots,\mathrm{det}(\gamma_m)) ,\quad N(c z+d)^{\mathbf{k}}=\prod_{i=1}^{m}\left(c_i z_{i}+d_i\right)^{k_i}.
\end{align*}
According to Koecher's principle \cite{Freitag}, each $f$ in $M_{\mathbf{k}}\left(\Gamma,\delta\right)$ has the following Fourier expansion at $\infty$
\[f(z)=\sum_{\nu \in \mathfrak{c}^{+}\cup\{0\}} a(\nu) e(\nu z),\]
where $e(\nu z)=\exp (2 \pi i \operatorname{Tr}(\nu z))$ and $\operatorname{Tr}(\nu z)=\sum_{i=1}^{m} \nu_i z_{i} $. The Fourier expansion of $f$ at other cusps can be defined and if the constant term $a(0)=0$ at each cusp, $f$ is called a cusp form and the corresponding space is denoted by $S_\mathbf{k}(\Gamma ,\delta)$. 

To construct twisted Siegel-Weil formulas, we also need to consider automorphic forms on $\GL_2(E)\backslash \GL_2(\A_E)$.
Denote $\widehat{\mathcal{O}}_E=\prod_\mathfrak{p}\mathcal{O}_\mathfrak{p}$, where $\mathfrak p$ runs over all nonzero prime ideals of $\mathcal{O}_E$.
Let $\mathcal{M}_{k}(\mathfrak{n},\delta)$ ($\mathcal{S}_{k}(\mathfrak{n},\delta)$) be the space of adelic Hilbert modular (cusp) forms with central character $\delta$ of weight $\mathbf{k}$ and level
\begin{align*}
K_{0}(\mathfrak{n})=\left\{\left(\begin{array}{ll}
a & b \\
c & d
\end{array}\right) \in \mathrm{GL}_{2}(\widehat{\mathcal{O}}_E): c \in \mathfrak{n} \widehat{\mathcal{O}}_E\right\}.
\end{align*}
For any $F\in \mathcal{M}_{\mathbf{k}}\left(\mathfrak{n},\delta\right)$, we have the Whittaker function
\begin{align*}
& W_F(g) := \int_{E\backslash \A_E} F(n(x)g) \psi_E (-x) dx,\quad g\in \GL_2(\A_E).
\end{align*}

Since $K_{0}(\mathfrak{n})$ has determinant $\widehat{\mathcal{O}}^{\times}_E$, $\mathrm{GL}_{2}(E) \backslash \mathrm{GL}_{2}(\mathbb{A}_E) / K_{0}(\mathfrak{n}) \mathrm{GL}_{2}^{+}(\mathbb{R})^{d}$ has cardinality $h^{+}$. 
Let $\{ t_{j} \}_{j =1}^{h^+}$ be finite ideles such that the fractional ideals $\mathfrak{c}_{j}:=t_{j}\mathcal{O}_E$ form a complete set of representatives of the narrow class group $\Cl^+(E)$ of $E$. 
Set $\Gamma_{j}=\Gamma_{0}\left(\mathfrak{c}_{j}, \mathfrak{n}\right)$. Shimura \cite{Shimura78} proved that there exists an isomorphism between adelic Hilbert modular forms and $h^+$-tuples of classical Hilbert modular forms, namely,
\begin{align}\label{classical and adelic}
\mathcal{M}_{\mathbf{k}}\left(\mathfrak{n},\delta\right) \cong \prod_{j=1}^{h^{+}} M_{\mathbf{k}}\left(\Gamma_{j},\delta\right)
\quad \text { and } \quad
\mathcal{S}_{\mathbf{k}}\left(\mathfrak{n},\delta\right) \cong \prod_{j=1}^{h^{+}} S_{\mathbf{k}}\left(\Gamma_{j},\delta\right).
\end{align}

\section{Orthogonal Shimura varieties and theta lifts}
In this section, we recall orthogonal Shimura varieties associated to quadratic spaces of signature $(2,2)$ and special subvarieties constructed from subspaces, and refer to \cite{Kudla97} for more details and general construction. 
Using the Weil representation on Schwartz functions, the integral against associated theta functions define Doi-Naganuma lift form $\SL_2$ to $O(2,2)$.
Then we formulate the twisted CM values of Borcherds forms in the spirit of \cite{Kudla03} and \cite{BKY12}.

\subsection{Orthogonal Shimura varieties}
Consider the rational quadratic space 
\begin{align*}
V_2:=\left\{\left.A=\left(\begin{array}{cc}
a & \lambda \\
\sigma(\lambda) & b
\end{array}\right) \right\rvert\, a, b \in \mathbb{Q}, \lambda \in F_2\right\}.
\end{align*}
with quadratic form $Q(A) = \det A$. The associated symmetric bilinear form $(x,y) :=Q(x+y)-Q(x)-Q(y)$.
Let $V_{2,\Z} = V_2 \cap M_{2\times 2}(\Oc_{F_2})$ and $H=\Gspin(V_{2,\Z})$.  
More explicitly, for any ring $R$,
\begin{align*}
H(R) \cong \{ \gamma\in \GL_2(\Oc_{F_2} \otimes_\Z R) : \det \gamma \in R^\times  \}
\end{align*}
acts on $V_{2,\Z}(R) := V_{2,\Z} \otimes R$ via 
\begin{align*}
\gamma \cdot A:= \det \gamma^{-1} \sigma(\gamma) A \gamma^t .
\end{align*}
Let $\Dc$ be the Grassmannian of negative oriented two dimensional subspaces of $V_2(\R)$.  
The hermitian symmetric space $\Dc$ has two connected components $\Dc = \Dc^+ \sqcup \Dc^-$ given by the two possible orientations.
Let $K$ be an open compact subgroup in $H(\A_f)$. 
The complex points of the associated Shimura variety are
\begin{align*}
X_K := H(\Q) \backslash \Dc \times H(\A_f) / K. 
\end{align*}
It's a complex analytic space of dimension $2$. 
One of its connected components is isomorphic to the Hilbert modular surface $X_\Gamma \cong \Gamma \backslash \h^2$, where $\Gamma = K \cap \SL_2(\Oc_{F_2})$. If $\Gamma = \SL_2(\Oc_{F_2})$, we denote $X_\Gamma$ by $X_{F_2}$.

The rational subspaces of $V_2$ induce natural subvarieties of $X_K$.
Let $V_1\subset V_2$ be a negative definite two dimensional subspace. 
Up to orientation, $V_1$ defines two points $z_{V_1}^\pm \in \Dc$.
The group $T_{V_1} := \Gspin (V_1)$ embeds into $H$ via acting trivially on $V_1^\perp$.
Denote by $K_{V_1} := K \cap T_{V_1}(\A_f)$, we obtain the 0-cycle
\begin{align}\label{special cycles}
Z(V_1) := T_{V_1}(\Q) \backslash \{ z_{V_1}^\pm \} \times T_{V_1}(\A_f) / K_{V_1}.
\end{align}
See \cite[Section 2]{BKY12} for more information.

Let $A \in V_2$ with $Q(A)>0$ and $H_A$ be its stabilizer in $H$. 
The hermitian symmetric space of $H_A$ gives the analytic divisor
\begin{align*}
    \Dc_A := \{ z\in \Dc : z\perp A \}
\end{align*}
in $\Dc$. Note that the intersection $\Dc_A \cap \h^2$ is
\begin{align*}
A^\perp = \left\{\left(z_1, z_2\right) \in \h^2 \mid a z_1 z_2+\lambda z_1+\sigma(\lambda) z_2+b=0\right\}.
\end{align*}
For $h\in H(\A_f)$, let $K_{h,A} = H_A(\A_f)\cap hKh^{-1}$ be the associated compact open subgroup of $H_A(\A_f)$.
Following \cite{Kudla97}, see also \cite[Section 2]{BEY21}, we define the associated divisor in $X_K$ as
\begin{align*}
Z(h,A) := H_A(\Q) \backslash \Dc_A \times H_A(\A_f)/K_{h,A} \rightarrow X_K : [z,\widetilde{h}] \mapsto [z,\widetilde{h}h].
\end{align*}
Given $m\in \Q^+$ and a $K$-invariant Schwartz function $\varphi\in \Sc (V_2(\A_f))$, we define
\begin{align}\label{special divisors}
Z(m,\varphi) := \sum_{h\in  H_A(\Q) \backslash H(\A_f) / K} \varphi(h^{-1}x) Z(h,A)
\end{align}
if there exists $A\in V_2(\Q)$ such that $Q(A)=m$. Otherwise, we set $Z(m,\varphi)=0$.

\subsection{CM points}
Define $\widetilde{D} = \Nm_{F_2/\Q} \mathfrak{d}_{K_4/F_2}$. 
Then $\sqrt{\widetilde{D}}\in \wf_2$ and $\sigma (\sqrt{\widetilde{D}}) = - \sqrt{\widetilde{D}}$ \cite[Equation (4.3)]{BY06}.
Consider $\wk_4$ as a $\Q$-quadratic space with quadratic form
\begin{align*}
q(x)=\tr_{\wf_2 / \mathbb{Q}} \frac{\Nm_{\wk_4/\wf_2} x}{\sqrt{\widetilde{D} }} =\frac{1}{\sqrt{\widetilde{D} }}(x \bar{x}-\sigma(x) \overline{\sigma(x)}),\quad x\in\wk_4.
\end{align*}
Note that $\Nm_{\wk_4/\wf_2} x=\frac{1}{2}\left(\tr_{\wf_2 / \mathbb{Q}} x\bar{x}+ \sqrt{\widetilde{D} } q(x)\right)$.
CM points associated to ideals in $K_4$ induce isomorphisms between $\wk_4$ and $V_2$, which we recall as follows.

Fix a non-zero and totally imaginary element $\xi_0 \in K_4$ such that $(\xi_0,\sigma (\xi_0)) \in \h^2$.
Let $\mathfrak{a}$ be a fractional ideal in $K_4$ and $\mathfrak{f}_{\mathfrak{a}} := \xi_0 \mathfrak{d}_{K_4/F_2} \mathfrak{a} \overline{\mathfrak{a}} \cap F_2$. 
Any CM point $z\in X_{F_2}$ of type $(K_4,\Phi)$ is represented by a pair $(\mathfrak{a},r)$, where $\mathfrak{a}$ has a decomposition $\mathfrak{a}=\mathcal{O}_{F_2} \alpha+\mathcal{O}_{F_2} \beta$ such that $z=(\frac{\alpha}{\beta},\sigma(\frac{\alpha}{\beta}))$,
and $r\in F_2^+$ satisfies $\mathfrak{f}_{\mathfrak{a}} = r \Oc_{F_2}$ \cite[Section 3]{BY06}. 
Define the map $\iota : V_2 \longrightarrow \wk_4$ by
\begin{align}\label{iso CM point}
\iota: V_2 & \longrightarrow \wk_4 , \quad
A  \mapsto(\sigma(\alpha)), \sigma(\beta)) A\binom{\alpha}{\beta},
\end{align}
which gives a $\mathbb{Q}$-isometry between quadratic spaces $(V_2, \det)$ and $\left(\wk_4, \frac{-q}{\Nm \mathfrak{a}} \right)$ by \cite[Proposition 4.3]{BY06}. 

Let $C(K_4)$ be the generalized ideal class group of $K_4$, whose elements are equivalence classes $[\mathfrak{b}, \xi]$ of pairs $(\mathfrak{b}, \xi)$, where $\mathfrak{b}$ is a fractional ideal of $K_4$ and $\xi \in F_2^{\times}$ satisfies $\Nm_{K_4 / F_2} \mathfrak{b}=\xi \Oc_{F_2}$. 
Two pairs $\left(\mathfrak{b}_1, \xi_1\right)$ and $\left(\mathfrak{b}_2, \xi_2\right)$ are equivalent if there is an element $\alpha \in K_4^{\times}$ such that $\mathfrak{b}_1=\alpha \mathfrak{b}_2$ and $\xi_1=\Nm_{K_4 / F_2}(\alpha) \xi_2$. 
Let $C_{+}(K_4)$ be the subgroup of classes $[\mathfrak{b}, \xi]$ such that $\xi$ is totally positive. 
For $(\mathfrak{b},\xi)\in C(K_4)$, its action on the CM point $(\mathfrak{a},r)$ is given by 
\begin{align*}
(\mathfrak{b},\xi) \bullet (\mathfrak{a},r) = (\mathfrak{b}^{-1}\mathfrak{a},\xi r). 
\end{align*}
The orbit of $z$ under the action of $C_+(K_4)$ modulo the group of units in $K_4$ is isomorphic to the complex points of the moduli stack of CM abelian surfaces with type $\Phi$ \cite[Proposition 6.2]{BKY12}.

\subsection{Twisted CM cycles}
For our purpose, as in \cite{BKY12}, we consider a smaller orbit defined by torus.
Let $\wf_2$-quadratic spaces 
\begin{align*}
W = (\wk_4,  \Nm_{\wk_4/\wf_2}),\quad  
W^+ = (\wk_4,  \frac{\Nm_{\wk_4/\wf_2}}{\sqrt{\widetilde{D}}}),\quad 
W^- = (\wk_4,  \frac{\Nm_{\wk_4/\wf_2}}{-\varepsilon\sqrt{\widetilde{D}}}),
\end{align*}
where $\varepsilon\in \wf_2^+$ such that $-\varepsilon$ is locally a norm at all finite places.
Then $W^+$ and $W^-$ are the neighborhood $\wf_2$-quadratic spaces at $\sigma$ and $\sigma^2$ of the admissible incoherent $\A_{\wf_2}$-quadratic space $\mathbb{W}$ associated to $W$ in the sense of \cite{BKY12}. 
The latter positive definite subspace of
\begin{align*}
W^+\otimes_{\wf_2} \R \cong \lb \wk_4\otimes_{\wf_2,\sigma} \R, \frac{\Nm_{\wk_4/\wf_2}}{-\sqrt{\widetilde{D}}} \rb  \oplus \lb \wk_4\otimes_{\wf_2,\sigma^2} \R, \frac{\Nm_{\wk_4/\wf_2}}{\sqrt{\widetilde{D}}} \rb.
\end{align*}
gives rise to CM points $z_+^\pm \in \Dc$ up to orientation under the isomorphism $\iota$.
Define $S^+ = \SO (W^+)$, and $T^+$ be the preimage of $\Res_{\wf_2/\Q} S^+$ in $\Gspin (V_2)$.
Then $T^+$ is a maximal torus and for any $\Q$-algebra $R$,
\begin{align*}
T^+(R) = \{ t\in (K_4 \otimes_\Q R )^\times : t\sigma^2(t) \in R^\times \}.
\end{align*}
There is an exact sequence
\begin{align}\label{torus so}
1 \longrightarrow \mathbb{G}_m \longrightarrow T^+ \mathop{\longrightarrow}\limits^{\nu} S^+ \longrightarrow 1, \quad \nu(t)=\frac{\Nm_{\Phi}(t)}{t \sigma^2(t)} 
\end{align}
\cite[Equation (6.5)]{BKY12}.
Let $K_{T^+} := T^+(\widehat{\Q}) \cap \SL_2(\widehat{\Oc}_{F_2})$. Then
\begin{equation} \label{eq:ZW1}
Z(W^+) := T^+(\mathbb{Q}) \backslash \{z^\pm_+\} \times T^+(\widehat\Q) / K_{T^+}
\end{equation}
is a 0-cycle on $X_{F_2}$ defined over $\wf_2$. 

\begin{remark}
$W^+$ is a canonical model of CM points in the sense of \cite[Proposition 6.11]{BKY12}.
Since $\sigma^2$ is the complex conjugation, Lemma 3.4 in \cite{BY06} shows that
\begin{align*}
\sigma^2 Z(W^+) = Z(W^+).
\end{align*}
Moreover, when $d_{K_4}=p^2 q$ with $ p \equiv q \equiv 1 (\bmod 4)$ being odd primes, the 0-cycle $Z(W^+)$ is a single Galois orbit of a CM point \cite[Lemma 6.1]{BY07}.
\end{remark}

The Galois conjugate of $Z(W^+)$ under the action of $\sigma$ is the  0-cycle 
\begin{equation} \label{eq:ZW2}
Z(W^-) := T^-(\mathbb{Q}) \backslash \{z^\pm_-\} \times T^-(\widehat\Q) / K_{T^-}
\end{equation}
similarly defined by the quadratic space $W^-$ \cite[Section 2]{BKY12}.
Together, they give us the CM cycle on $X_{F_2}$
\begin{align*}\label{eq:ZW}
Z(W) := Z(W^+) + Z(W^-),
\end{align*}
which is defined over $\Q$, see also \cite[Section 3]{BY06}. Note that $S^+ = S^-$ and we will denote them both as $S= \SO(W)$ without confusion.

The embedding
\begin{align}
T^\pm(\mathbb{Q}) \backslash T^\pm(\widehat\Q) / K_{T^\pm} \rightarrow C_+(K_4) : \quad
t \mapsto (t\Oc_{K_4},\xi),
\end{align}
where $\xi \in \Q^+$ is the unique element such that $t \bar{t} \mathbb{Z}=\xi \mathbb{Z}$, is well-defined and an injective homomorphism \cite[Equation (2.6)]{HY12}. Denote its image as $C(T^\pm)$.
Then the 0-cycle $Z(W)$ coincides with the union of orbits of the CM point $(\mathfrak{a},r)$ under the action of 
$C(T^\pm)$ as in \cite[Corollary 6.6]{BKY12}.
Now we can rewrite the CM cycle as
\begin{equation} \label{ZW}
Z(W) = \sum_{\gamma \in T^+(\mathbb{Q}) \backslash T^+(\widehat\Q) / K_{T^+}}  \gamma \bullet z_+^\pm 
+\sum_{\eta \in T^-(\mathbb{Q}) \backslash T^-(\widehat\Q) / K_{T^-}}  \eta \bullet z_-^\pm
\end{equation}
Any character $\chi$ on $S$ naturally lifts to a character $\chi_{T^\pm}$ on $T^\pm$ via the projection in the exact sequence \eqref{torus so}.
We define the twisted CM cycles by
\begin{align}\label{twisted ZW}
\begin{split}
Z_\chi(W^+)= \sum_{\gamma \in T^+(\mathbb{Q}) \backslash T^+(\widehat\Q) / K_{T^+}} \chi_{T^+}(\gamma) \gamma \bullet z_+^\pm,\\
Z_\chi(W^-)=\sum_{\eta \in T^-(\mathbb{Q}) \backslash T^-(\widehat\Q) / K_{T^-}} \chi_{T^-}(\eta) \eta \bullet z_-^\pm.
\end{split}
\end{align}
In the following, we will fix $\chi$ to be the character defined in Lemma \ref{wk4 character}.

View $\sigma$ as an element in $\operatorname{Aut}(\mathbb{C})$, the Tate cocycle $f_{\Phi}(\sigma) \in \A_{K_4,f}^{\times} / K_4^{\times}$\cite[Section 4]{milne2007fundamental}. 
Let $t \in \A_{K_4,f}^{\times}$ be an element of the coset $f_{\Phi}(\sigma)$ and $\mathfrak{a}=t \Oc_{K_4}$ be the fractional ideal generated by $t$. 
More explicitly, let $\chi_{\mathrm{cyc}}$ be the cyclotomic character
$$\chi_{\mathrm{cyc}}: \operatorname{Aut}(\mathbb{C}) \rightarrow \operatorname{Gal}(\overline{\mathbb{Q}} / \mathbb{Q}) \rightarrow \widehat{\mathbb{Z}}^{\times} $$
satisfying $\Art_{\mathbb{Q}}\left(\chi_{\mathrm{cyc}}(\eta)\right)=\eta \mid_{\mathbb{Q}^{\mathrm{ab}}}$ \cite[Lemma 3.4]{milne2007fundamental}. 
Then $t \bar{t}=\chi_{\mathrm{cyc}}(\sigma) \cdot \xi$ \cite[Proposition 4.6]{milne2007fundamental}, where $\xi \in F_2^{\times}$ satisfies $\xi \Oc_{F_2} = N_{K_4 / F_2} \mathfrak{a}$. 
The pair $\mathbf{c}=(\mathfrak{a}, \xi)$ defines a class in $C(K_4)$, and the action of $\sigma$ on CM points is given by $\mathbf{c}$. 
Similarly, the action of $\sigma^2$ is given by $(\Oc_{K_4},-1)$.
So one can extend the character $\chi_{T^\pm}$ on $T^\pm \cong C(T^\pm)$ to the group $\langle C(T^\pm) , \mathbf{c}\rangle$ by defining $\chi_{T^\pm}(\mathbf{c}) = \pm 1$.

\subsection{Borcherds lift}
Let $L$ be an integral lattice in $V_2$. 
Then the stabilizer of $\widehat{L} := L\otimes \widehat{\Z}$ in $H(\A_f)$ is an open compact subgroup $K=K_L$. 
We denote by $L^{\vee} \subset V$ the dual lattice of $L$ and $\widehat{L}^{\vee}:=L^{\vee} \otimes \widehat{\mathbb{Z}}$.
For $m \in \mathbb{Q}, \mu \in L^{\vee} / L$, we define
\begin{align*}
L_{m, \mu}:=\{\lambda \in \mu+L : Q(\lambda)=m\}.
\end{align*}
Let $M_{k, L}^{!}$ be the space of weakly holomorphic modular forms valued in $\mathbb{C}\left[L^{\vee} / L\right]$ of weight $k \in \frac{1}{2} \mathbb{Z}$ with respect to the representation $\rho_L$ on $\mathrm{SL}_2(\mathbb{Z})$ \cite[Section 3]{BF04}.
For $f(\tau)=\sum\limits_{m \in \mathbb{Q}, \mu \in L^{\vee} / L} c(m, \mu) q^m \mathfrak{e}_\mu \in M_{k, L}^{!}$, we denote by
$$\operatorname{prin}(f):=\sum_{m \in \mathbb{Q}_{<0}, \mu \in L^{\vee} / L} c(m, \mu) q^m \mathfrak{e}_\mu$$
the principal part of $f$.

For $z\in \Dc$ and $x\in V_2(\Q)$, let $(x, x)_z=-\left(x_z, x_z\right)+\left(x_{z^{\perp}}, x_{z^{\perp}}\right)$. 
We define the associated Gaussian by 
\begin{align*}
\varphi_{\infty}(x, z)=\exp(-\pi(x, x)_z) .
\end{align*}
For a $K$-invariant Schwartz function $\varphi \in \Sc(V_2(\widehat{\mathbb{Q}}))^K$ and $\tau = u+iv \in \mathbb{H}$, the theta function
\begin{equation*}
\theta(\tau, z, h, \varphi)= \sum_{x \in V_2(\mathbb{Q})} \omega \left(g_\tau\right) \varphi_{\infty}(x, z) \varphi\left(h^{-1} x\right),\quad
g_\tau = n(u)m(\sqrt{v}),
\end{equation*}
is an automorphic function of $[z, h] \in X_K$. 
Let $\varphi_\mu$ be the characteristic function of $\mu+\widehat{L}$. The vector valued theta function
$$\Theta_L(\tau, z, h):= \sum_{\mu \in L^{\vee} / L} \theta \left(\tau, z, h, \varphi_\mu \right) \mathfrak{e}_\mu$$
is a non-holomorphic modular form of weight 0 in $\tau$.
Let $r\in\N$ and $R_{\tau}^r$ be the iterate raising operator. 
The regularized theta lift of $f\in M_{-2r, L}^{!}$ against $\Theta_L$ has integral representation \cite[Equation (2.41)]{BEY21}
\begin{align}\label{theta lift}
\begin{split}
\Phi_f^r(z, h) & =(4 \pi)^{-r} \lim _{T \rightarrow \infty} \int_{\mathcal{F}_T}\left\langle R_\tau^r f(\tau), \overline{\Theta_L(\tau, z, h)}\right\rangle d \mu(\tau)\\
& =(-4 \pi)^{-r} \lim _{T \rightarrow \infty} \int_{\mathcal{F}_T}\left\langle f(\tau), \overline{R_\tau^r \Theta_L(\tau, z, h)}\right\rangle d \mu(\tau),
\end{split}
\end{align}
where $\mathcal{F}_T$ is the truncated fundamental domain of $\operatorname{SL}_2(\mathbb{Z}) \backslash \mathbb{H}$ at height $T>1$. 
It has logarithmic singularity along the special divisor
$$Z_f:=\sum_{m>0, \mu \in L^{\vee} / L} c(-m, \mu) Z(m, \mu)$$
on $X_K$, where $Z(m, \mu):=Z\left(m, \varphi_\mu\right)$. 

\subsection{Twisted CM values}\label{twisted CM values}
For a Schwartz function $\varphi \in S(V_2(\widehat{\mathbb{Q}}))$, let $\phi^{+}\in \Sc (W^+)(\A_{\wf_2,f})$ be the image of $\varphi$ under the isomorphism $\iota$.
By construction, $-\varepsilon=\Nm\ \beta_0$ for some $\beta_0\in\A_{\wk_4,f}$.
Consider the isometry
\begin{align*}
W^+(\A_{\wf_2,f}) \cong W^-(\A_{\wf_2,f}) :\quad x\mapsto  \sigma^2(\beta_0) x.
\end{align*}
Let $\phi^- \in \Sc (W^-)(\A_{\wf_2,f})$ be the image of $\phi^+$ under this isometry, namely, $\phi_f^+(x) = \phi_f^-(\sigma^2(\beta_0) x)$. 
Denote by $\phi_\infty^\pm$ the Gaussian of the quadratic spaces $W^\pm$.
Then the Hilbert theta functions defined in \cite[Equation (4.3)]{BKY12} become
\begin{align*}
\theta_{\wk_4} (g,h,\phi_f^\pm \phi_\infty^\pm) := \sum_{x\in\wk_4} (\omega(g,h) \phi_f^\pm \phi_\infty^\pm )(x),
\quad (g, h)\in R(W^\pm)(\A_{\wf_2})
\end{align*}
by \cite[Lemma 4.1]{BKY12}.

As in \cite[Section 4]{BKY12}, the pull back of $\theta(\tau, z, h, \varphi)$ to $Z_\chi(W^\pm)$ coincides with the pull back of the Hilbert theta function $\theta_{\wk_4} (g,h,\phi^\pm \phi_\infty^\pm )$ via the diagonal embedding $\h \rightarrow \h^2$. 
Consequently, we can interpret the twisted CM values of $\theta (\tau,z,h,\varphi)$ at $Z_\chi (W^\pm)$ as 
\begin{align}\label{CM values torus}
\begin{split}
\theta (\tau,Z_\chi(W^\pm),\varphi) &= \frac{\deg Z(W)}{4} 
\int_{[\SO(W)]} \theta_{\wk_4}(g_\tau^\Delta,h,\phi_f^\pm \phi_\infty^\pm) \chi(h) dh
\end{split}
\end{align}
similarly as \cite[Lemma 4.4]{BKY12}, 
where $[\SO(W)] = \SO(W)(\wf_2) \backslash \SO(W)(\A_{\wf_2})$.

\begin{remark}
If we take a different isometry $W^+(\A_{\wf_2,f}) \rightarrow W^-(\A_{\wf_2,f})$ in \cite[Equaiton (2.9)]{BKY12}, we might get different Hilbert theta functions and hence different twisted CM values as they depend on the Schwartz functions $\phi_f^\pm$.
However, due to the the Siegel section map, the CM values of theta functions are independent on the choice of the isometry $W^+(\A_{\wf_2,f}) \rightarrow W^-(\A_{\wf_2,f})$ by \cite[Lemma 4.2]{BKY12}.
\end{remark}

\section{Theta integrals and Whittaker functions}\label{theta integrals}
In this section, we calculate the Whittaker functions of theta integrals and explore their properties, which will be used to prove twisted Siegel-Weil formulas at next section.

\subsection{Twisted theta integral}\label{twisted theta integral}
Recall $W$ is the quadratic space $\wk_4$ over $\wf_2$ with quadratic form $q := \Nm_{\wk_4/\wf_2}$, and the additive character $\psi_{\wf_2}$ is the composition of the standard additive character of $\psi$ on $\A$ and $\tr_{\wf_2/\Q}$.
Denote by 
$$\GL_2(\A_{\wf_2})^+ := \{ \gb \in \GL_2(\A_{\wf_2}) \mid \det \gb \in \nu (\GO (W)(\A_{\wf_2})) \}.$$
Let $\phi\in \Sc(W(\A_{\wf_2}))\cong \Sc (\A_{\wk_4})$ and $\chi$ be a Hecke character on $\A_{\wk_4}^\times$. 
The twisted theta integral against $\chi$ is
\begin{align}\label{theta wk4}
\theta_{\wk_4,\chi} (\gb,\phi) := \int_{[\SO(W)]} \theta_{\wk_4} (\gb,\beta h,\phi) \chi(\beta h) dh,\quad \gb \in \GL_2(\A_{\wf_2})^+,
\end{align}
where $\beta\in\A_{\wk_4}^\times$ satisfies $\nu (\beta) = \det \gb$.
Note that this integral doesn't depend on the choice of $\beta$.
We can naturally extend $\theta_{\wk_4,\chi} (\gb,\phi)$ to a function on the whole $\GL_2(\A_{\wf_2})$ by defining it to be zero outside $\GL_2(\A_{\wf_2})^+$.
It's easy to verify that $\theta_{\wk_4,\chi} (\gb,\phi)$ is an automorphic form on $\GL_2(\wf_2)\backslash\GL_2(\A_{\wf_2})$ with central character $\chi_{\wk_4/\wf_2} \cdot \chi$.
Let $W_{\wk_4,\chi}(\gb , \phi)$ be the Whittaker function of $\theta_{\wk_4,\chi} (\gb,\phi)$. If $\phi$ is clear in the context, we will omit it form the notation and write as $W_{\wk_4,\chi}(\gb)$.

\begin{proposition}\label{whittaker wk4}
Assume that the Schwartz function $\phi$ is a pure tensor $\phi=\prod_\fp \phi_\fp$. Then
\begin{enumerate}
\item   For $\gb \in \GL_2(\mathbb{A}_{\wf_2})$,  $W_{\wk_4,\chi}(\gb , \phi)$ decomposes into a product of local Whittaker functions on $\GL_2 ( \wf_{2,\mathfrak{p}} )$ given by
$$W_{\wk_4,\chi, \mathfrak{p}}(\gb_\mathfrak{p} , \phi_\mathfrak{p}) =
\int_{\wk^1_{4,\mathfrak{p}}} L\left(\beta_\mathfrak{p}\right) \omega \left(\gb_{2,\mathfrak{p}}\right) 
\phi_\mathfrak{p} \left(h_\mathfrak{p}^{-1}\right) \chi_\mathfrak{p} \left(\beta_\mathfrak{p} h_\mathfrak{p}\right) d h_\mathfrak{p}.$$
\item For $\xi \in \wf_2^{\times}, \gb \in \GL_2(\mathbb{A}_{\wf_2})$, we have
$$W_{\wk_4,\chi} \left(\pmat{\xi}{}{}{1} \gb , \phi\right)=0 \quad \text { if }\ \xi \notin \Nm_{\wk_4/\wf_2} \lb\wk_4^{\times}\rb.$$
\end{enumerate}
\end{proposition}
\begin{proof}
(1) The Whittaker function of $\theta_{\wk_4,\chi} (\gb,\phi)$ is
$$W_{\wk_4,\chi} (\gb , \phi)=\int_{\wf_2 \backslash \mathbb{A}_{\wf_2} } \theta_{\wk_4,\chi} (n(x)\gb,\phi) \psi_{\wf_2}(-x) dx.$$
Substituting the theta integral \eqref{theta wk4} and interchanging the order of integration,
$$W_{\wk_4,\chi} (\gb , \phi)= \int_{\wk_4^1 \backslash \mathbb{A}_{\wk_4^1} }
\int_{\wf_2 \backslash \mathbb{A}_{\wf_2}} \theta_{\wk_4} (n(x)\gb,\beta h,\phi) \psi_{\wf_2}(-x) d x \chi(\beta h) d h.$$
The inner integral is 
\begin{align*}
\int_{\wf_2 \backslash \mathbb{A}_{\wf_2}} \theta_{\wk_4} (n(x)\gb,\beta h,\phi) \psi_{\wf_2}(-x)  d x 
& =\sum_{t \in \wk_4} \int_{\wf_2 \backslash \mathbb{A}_{\wf_2}} 
\lb L(\beta) \omega\left(\gb_2\right) \phi \rb\left(h^{-1} t\right) \psi_{\wf_2}(x q(t)-x) d x \\
& =\sum_{t \in \wk_4^1} \lb L(\beta) \omega\left(\gb_2\right) \phi \rb \left(h^{-1} t\right).
\end{align*}
Since $\chi$ is trivial on $\wk_4^\times$, we obtain
$$W_{\wk_4,\chi} (\gb , \phi)= \int_{\A_{\wk_4^1}} 
\lb L\left(\beta\right) \omega \left(\gb_{2}\right) \phi \rb \left(h^{-1}\right) \chi\left(\beta h\right) d h.$$
This shows that $W_{\wk_4,\chi}$ is a product of the local Whittaker functions.

(2) Using the fact that $\theta_{\wk_4,\chi} (\gb,\phi)$ is left $\GL_2(\wf_2)$-invariant, as before, the inner integral of $W_{\wk_4,\chi} (\pmat{\xi}{}{}{1}\gb , \phi)$ is
\begin{align*}
&\int_{\wf_2 \backslash \mathbb{A}_{\wf_2}} 
\theta_{\wk_4} (n(x) \pmat{\xi}{}{}{1}  \gb,\beta h,\phi) \psi_{\wf_2}(-x)  d x \\
=&\int_{\wf_2 \backslash \mathbb{A}_{\wf_2}} 
\theta_{\wk_4} ( n(\xi^{-1}x) \gb,\beta h,\phi) \psi_{\wf_2}(-x)  d x \\
=& \sum_{t \in \wk_4} \int_{\wf_2 \backslash \mathbb{A}_{\wf_2}} 
\lb L(\beta) \omega\left(\gb_2\right) \phi \rb \left(h^{-1} t\right) \psi_{\wf_2}(x \xi^{-1} q(t)-x) d x,
\end{align*}
which vanishes if $\xi \notin\Nm_{\wk_4 / \wf_2}\left(\wk_4^\times\right)$ and proves part (2).    
\end{proof}

Similarly, the constant term of $\theta_{\wk_4,\chi} (\gb,\phi)$ is 
\begin{align*}
W_{\wk_4,\chi} (\gb, \phi)
:=\int_{\wf_2 \backslash \mathbb{A}_{\wf_2} } 
\theta_{\wk_4,\chi} (n(x)\gb,\phi) dx
= \lb L\left(\beta\right) \omega \left(\gb_{2}\right) \phi \rb \left(0\right) \chi(\beta) \int_{\wk_4^1\backslash\A_{\wk_4^1}} \chi\left( h\right) d h.
\end{align*}
Therefore, if $\chi$ is a non-trivial character on $\wk_4^1\backslash\A_{\wk_4^1}$, $\theta_{\wk_4,\chi} (\gb,\phi)$ is a Hilbert cusp form.

\subsubsection{Values at $T^1$ of Whittaker functions}
Let $\mathfrak{p}\subset\wf_2$ be a prime ideal.
If $\mathfrak{p}=\mathfrak{P}_1\mathfrak{P}_2$ is split in $\wk_4$, $\frob_{\mathfrak{P}_1}=(1,0,1)^{i_\fp}, \frob_{\mfp_2} = (1,0,1)^{5 i_\fp} \in \Gal(M_{32}/\wk_4)$, $i_\fp\in \Z/8\Z$, and hence $\frob_{\mathfrak{P}_1}\frob_{\mathfrak{P}_2}^{-1}= (1,0,1)^{4i_\fp}$.
When $\mathfrak{p}=\mathfrak{P}$ is inert in $\wk_4$, $\frob_{\mathfrak{p}}=(2,0,0)(1,0,1)^{i_\fp} \in \Gal(M_{32}/\wf_2)$. 
In this case, $\frob_{\mathfrak{P}} \in \Gal(M_{32}/\wk_4)$ equals $(i_\fp,i_\fp,0)$ if $i_\fp$ is even and $(i_\fp+2,i_\fp+2,0)$ if $i_\fp$ is odd.

\begin{proposition}\label{value T1 K4}
For any prime ideal $\mfp \subset\wk_4$, let $\phi_0 \in \Sc\left(\mathbb{A}_{\wk_4}\right)$ be the Schwartz function with local components 
\begin{align*}
\phi_{0,\mathfrak{P}}(x)= 
\begin{cases}
1_{\mathcal{O}_{\wk_{4,\mathfrak{P}}}}(x) & \text { if } \mathfrak{P} \text { is non-archimedean, } \\ 
\exp (-2\pi x\bar{x}) & \text { if } \mathfrak{P} \text { is archimedean. }  
\end{cases}
\end{align*}
Let $\chi$ be the character as in Lemma \ref{wk4 character}. 
For $\alpha\in \A_{\wf_2}^\times$, the values of local Whittaker functions in Proposition \ref{whittaker wk4} at $t(\alpha)$ are
\begin{align*}
\frac{W_{\wk_4,\chi, \mathfrak{p}}(t(\alpha))} 
{|\alpha|_{\wf_{2,\fp}}^{\frac{1}{2}}  
}= 
\begin{cases}
0 & \text { if }|\alpha|_\mathfrak{p}>1, i.e., \alpha\notin \Oc_{\wf_\fp} \\
\zeta_8^{i_\fp v_\fp(\alpha)} \lb v_{\mathfrak{p}}(\alpha)+1 \rb  & \text { if } \mathfrak{p}=\mathfrak{P}_1\mathfrak{P}_2 \text { is split and $i_\fp$ is even}\\
\zeta_8^{i_\fp v_\fp(\alpha)} \lb (-1)^{v_\mathfrak{p}(\alpha)} +1 \rb /2  & \text { if } \mathfrak{p}=\mathfrak{P}_1\mathfrak{P}_2 \text { is split and $i_\fp$ is odd}\\
\zeta_8^{i_\fp v_\fp(\alpha)} \lb (-1)^{v_\mathfrak{p}(\alpha)} +1 \rb /2  & \text { if } \mathfrak{p} \text { is inert and $i_\fp$ is even} \\
\zeta_8^{- i_\fp v_\fp(\alpha)} \lb (-1)^{v_\mathfrak{p}(\alpha)} +1 \rb /2  & \text { if } \mathfrak{p} \text { is inert and $i_\fp$ is odd} \\
( 1+\chi_{\wk_4/\wf_2}(\alpha_\fp) ) /2  & \text { if } \mathfrak{p}=\mathfrak{P}^2 \text { is ramified } \\
\frac{1+\sgn(\alpha)}{2}\exp (-2\pi \alpha) & \text { if } \mathfrak{p} \text { is archimedean. }  
\end{cases}
\end{align*}
\end{proposition}
\begin{remark}
Note that, in the previous formula, $W_{\wk_4,\chi}(t(\alpha))$ is zero if $\alpha\notin \Nm_{\wk_4/\wf_2} \A_{\wf_2}^\times$.
And $\phi_{0,\infty}$ gives rise to the Gaussian $\phi_\infty^{1,1}$ of weight $(1,1)$.
\end{remark}
\begin{proof}
In the proof, we briefly denote $\phi_0$ as $\phi$. 
Suppose $\alpha=\Nm\beta, \beta\in \A_{\wk_4}$. The value $W_{\wk_4,\chi}(t(\alpha))$ is
\begin{align*}
\int_{\A_{\wk_4^1}}  \omega \left(\pmat{\alpha}{}{}{1}, \beta \right) \phi\left(h^{-1}\right) \chi\left(\beta h\right) d h
=& |\alpha|_{\A_{\wf_2}}^{\frac{1}{2}} \chi_{\wk_4/\wf_2}(\alpha) \chi\left(\beta \right)
\int_{\A_{\wk_4^1}}   \phi\left(\alpha \beta^{-1} h^{-1}\right) \chi (h) d h.
\end{align*}
(1) When $\mathfrak{p}=\mathfrak{P}_1\mathfrak{P}_2$ is split in $\wk_4$, $\frob_{\mathfrak{P}_1}\frob_{\mathfrak{P}_2}^{-1}= (1,0,1)^{4i_\fp}$ and $\psi (1,0,1)^{4i_\fp}=(-1)^{i_\fp}$.
Let $\beta_\fp = (1,\alpha_\fp) \in \wk_{4,\fp} \cong \wf_{2,\fp} \times \wf_{2,\fp}$. Then
\begin{align*}
&\int_{\wk_{4,\mathfrak{p}}^1}  \phi_\mathfrak{p}\left(\alpha_\fp \beta_\fp^{-1} h^{-1}\right) \chi_\fp (h) d h
= \int_{\wf_{2,\mathfrak{p}}^\times}  \phi_\mathfrak{p}\left(\alpha_\fp h,h^{-1}\right) \chi_\fp (h,h^{-1})  d^\times h\\
=& \int_{\alpha_\fp^{-1}\mathcal{O}_{\wf_2,\mathfrak{p}}-\mathfrak{p}\mathcal{O}_{\wf_2,\mathfrak{p}}} (-1)^{i_\fp v_\fp(h)}   d^\times h
= \begin{cases}
 v_\mathfrak{p}(\alpha) +1  & \text{ if $i_\fp$ is even and } v_\mathfrak{p}(\alpha) \geq 0\\
\lb (-1)^{v_\mathfrak{p}(\alpha)} +1 \rb /2 & \text{ if $i_\fp$ is odd and } v_\mathfrak{p}(\alpha) \geq 0\\
0 & \text{ if } v_\mathfrak{p}(\alpha) < 0
\end{cases}.
\end{align*}
(2) When $\mathfrak{p}=\mathfrak{P}$ is inert in $\wk_4$, let $\varpi_\mathfrak{p}$ be the uniformizer of $\wf_{2,\mathfrak{p}}$. 
Suppose $\alpha_\fp=\varpi_\mathfrak{p}^{2i} u, u\in \Oc_{\wf_2}^\times,i\in\Z$, and $\beta_\fp =\varpi_\mathfrak{p}^{i} w \in \wk_{4,\mathfrak{P}},w\in \Oc_{\wk_4}^\times,$ such that $\Nm\beta_\fp=\alpha_\fp$. Then 
\begin{align*}
&\int_{\wk_{4,\mathfrak{p}}^1}  \phi_\mathfrak{p}\left(\alpha_\fp \beta_\fp^{-1} h^{-1}\right) d h
= \int_{\wk_{4,\mathfrak{p}}^1}  \phi_\mathfrak{p}\left( \varpi_\mathfrak{p}^{i} uw^{-1} h^{-1} \right)  d h
=
\begin{cases}
1, & \text{if } i\geq 0\\
0, & \text{otherwise }
\end{cases}.
\end{align*}
(3) When $\mathfrak{p}=\mathfrak{P}^2$ is ramified in $\wk_4$, let $\varpi_\mathfrak{p}$ be the uniformizer of $\wf_{2,\mathfrak{p}}$ and $\pi_\mathfrak{P}$ be the uniformizer of $\wk_{4,\mathfrak{P}}$. 
Suppose $\alpha_\fp=\varpi_\fp^{i} u$ and $\beta_\fp =\pi_\mathfrak{P}^{i} w \in \wk_{4,\mathfrak{P}}$ such that $\Nm\beta_\fp=\alpha_\fp$. Then
\begin{align*}
&\int_{\wk_{4,\mathfrak{p}}^1}  \phi_\mathfrak{p}\left(\alpha_\fp \beta_\fp^{-1} h^{-1}\right) d h
= \int_{\wk_{4,\mathfrak{P}}^1}  \phi_\mathfrak{p}\left( \pi_\mathfrak{P}^{i} uw^{-1} h^{-1} \right)  d h
=
\begin{cases}
1, & \text{if } i\geq 0\\
0, & \text{otherwise }
\end{cases}.
\end{align*}
(4) When $\mathfrak{p}=\infty$, $\wk_{4,\fp}\cong \C$ and $\wk_{4,\fp}^1 \cong S^1$ is the unit circle.
Suppose $\alpha_\fp = \beta_\fp \overline{\beta_\fp}$. Then
\begin{align*}
&\int_{\wk_{4,\mathfrak{p}}^1}  \phi_\mathfrak{p}\left(\alpha_\fp \beta_\fp^{-1} h^{-1}\right) d h
= \int_{S^1}  \phi_\fp \left( \overline{\beta_\fp} h^{-1} \right)  d h = \exp (-2\pi \alpha).
\end{align*}
Putting these cases and $\chi(\beta)$ together finish the proof.
\end{proof}

In general, it's easy to see that the value of Whittaker functions of $\theta_{\wk_4,\chi}(\gb,\phi)$ lie in the subfield $\Q(i,\phi)$ of $\mathbb{C}$ generated by the values $\phi(x)$, $x \in W(\A_{\wf_2})$ and $i$. 
Moreover, when $\gb\in \SL_2(\A_{\wf_2})$, for $\alpha\in \wf_2^\times$, the $\alpha$-th Whittaker coefficient
\begin{align}\label{sl2 whittaker}
W_{\wk_4,\chi}(t(\alpha),\phi) =& \int_{\A_{\wk_4^1}}   \phi\left(\alpha \beta^{-1} h^{-1}\right) \chi (h) d h
\end{align}
lies in $\Q(\phi)$. 
Eisenstein series also have similar properties, see \cite[Proposition 4.6]{BKY12}.

\subsection{Hecke's integral}
We consider the the quadratic space $U=F$ over $\Q$ with quadratic form $a_2\Nm_{F/\Q}$, $a_2\in\Q$.
Similarly as in the previous section, for $\Xi\in \Sc(U)(\A)$, the theta lift of the Hecke character $\rho$ in Lemma \ref{matching central characters} on $F^\times \backslash \A_{F}^\times$ is
\begin{align*}
\theta_{\rho} (g,\Xi) := \int_{F^1\backslash\A_{F^1}} \theta (g,\lambda h,\Xi) \rho(\lambda h) dh,\quad g\in\GL^+_2(\A),
\end{align*}
where $F^1$ is the kernel of $\Nm_{F/\Q}$, $\lambda\in\A_F^\times$ such that $\det g = \Nm_{F/\Q} \lambda$, and
\begin{align*}
\GL^+_2(\A) = \{ g\in\GL_2(\A) : \exists\ \lambda\in \A_F^\times \text{ such that } \det g = \Nm_{F/\Q} \lambda \}.
\end{align*}
We also extend $\theta_{\rho} ( g ,\Xi)$ to an automorphic form on $\GL_2(\A)$ by defining it to be 0 outside $\GL^+_2(\A)$.
It's easy to verify that $\theta_{\rho} ( g ,\Xi)$ is an elliptic cusp form and has central character $\chi_{F/\Q} \cdot \rho $.
Such theta integral is called Hecke's integral, for more information, see \cite{BLY22} and \cite{CL20}. 
As in the previous section, the Whittaker functions have the following properties, see also \cite[Proposition 4.1.2]{popa06}.

\begin{proposition}\label{whittaker wf2}
Assume that the Schwartz function $\Xi$ is a pure tensor $\Xi=\prod_p \Xi_p$, and let $W_{F,\rho}(g , \Xi)$ be the Whittaker coefficient of $\theta_{\rho} (g,\Xi)$. Then
\begin{enumerate}
\item   
For $g \in \GL_2(\mathbb{A})$,  $W_{F,\rho}(g,\Xi)$ decomposes into a product of local Whittaker functions on $\GL_2 (\Q_p)$ given by
$$W_{F,\rho}(g_p , \Xi_p) =
\int_{F^1_p} \lb L\left(\lambda_p\right) \omega \left(g_{2,p}\right) \Xi_p \rb\left(h_p^{-1}\right) \rho_p \left(\lambda_p h_p \right) d h_p.$$
\item 
For $\xi \in \Q^{\times}$ and $g \in \GL_2(\mathbb{A})$, we have
$$W_{F,\rho} \left(\pmat{\xi}{}{}{1} g,\Xi\right)=0 \quad \text { if }\ \xi \notin \Nm_{F / \Q}\left(F^{\times}\right) .$$
\end{enumerate}
\end{proposition}
\begin{proof}
The proof is the same as Proposition \ref{whittaker wk4}, so we omit it here.
\end{proof}

\subsection{Doi-Naganuma lift}
Recall that $V$ is the quadratic space $\Q^2\oplus\wf_2$ with quadratic form $Q(x,y,\nu) := a_1xy-a_1\Nm \nu$ for $a_1\in\Q$.
For $\varphi\in \Sc (V(\A))$, the Doi-Naganuma lift of Hecke's integral is
\begin{align*}
I(\hb,\varphi,\Xi,\rho) = \int_{[G]} \theta_V ( g d(\nu(\hb)),\hb,\varphi) 
\int_{F^1\backslash \A_{F^1}} \theta ( g d(\nu(\hb)),  h \lambda,\Xi) \rho ( h \lambda ) dh dg,\ \hb\in \GO(V)(\A),
\end{align*}
where $[G] = G(\Q)\backslash G(\A)$ and $\lambda\in \A_{F}^\times$ satisfies $\nu(\lambda)=\nu(\hb)$. 
This definition doesn't depend on the choice of $\lambda$. 
Using the embedding $i$ \eqref{embedding GL2}, $I(\hb,\varphi,\Xi,\rho)$ is an adelic Hilbert modular form on $\GL_2(\A_{\wf_2})$ and left $\GL_2(\wf_2)$-invariant as in \cite[Equation (2.52)]{HL25}.
In the following, we will assume $\hb\in \GL_2(\A_{\wf_2})$, and by a slightly abuse of notation, we briefly write $I(\hb,\varphi,\Xi,\rho) = I(i(\hb),\varphi,\Xi,\rho)$.
It's easy to show that the central character of $I(\hb,\varphi,\Xi,\rho)$ is $(\chi_{\wf_2/\Q} \chi_{F/\Q} \rho^{-1}) \circ \Nm_{\wf_2/\Q}$.

The Whittaker function $W(\hb) = W(\hb,\varphi,\Xi,\rho)$ of $I(\hb,\varphi,\Xi,\rho)$ is
\begin{align*}
 & \int_{\wf_2\backslash \A_{\wf_2}} \int_{[G]} \theta_{V} ( g d(\nu(\hb)),n(x)\hb,\varphi) 
\int_{F^1\backslash \A_{F^1}} \theta ( g d(\nu(\hb)),  h \lambda,\Xi) \rho ( h \lambda ) dh dg \psi_{\wf_2} (- x) dx\\
&= \int_{F^1\backslash \A_{F^1}} \int_{\wf_2\backslash \A_{\wf_2}} 
\int_{[G]} \theta_{V} ( g d(\nu(\hb)),n(x)\hb,\varphi) \theta ( g d(\nu(\hb)), h \lambda,\Xi) dg \psi_{\wf_2} (- x) dx  \rho (h\lambda) dh.
\end{align*}
A set of representatives of $\Q^2$ under the right multiplication of $\SL_2(\Q)$ is given by $(0,0)$ with stabilizer $\SL_2(\Q)$ and $(0,1)$ with stabilizer $N(\Q)$. 
Applying the partial Fourier transform, we can rewrite
\begin{align*}
&\theta_{V} ( g d(\nu(\hb)),n(x) \hb,\varphi) 
= \theta_{V} (1,1,\omega (g)L\lb n(x) \hb \rb \varphi)\\
=& \sum_{\nu \in \wf_2} \lb \omega_0(g) \Fc (L( n(x)\hb )\varphi) \rb \lb (0,0), \nu  \rb 
+\sum_{\nu \in \wf_2} \sum_{\substack{\gamma\in N(\Q) \backslash \SL_2(\Q)}} 
\lb \omega_0(g) \Fc (L( n(x)\hb )\varphi) \rb \lb (0,1)\gamma g, \nu  \rb,
\end{align*}
where $\omega_0$ is the Weil representation on the quadratic space $\wf_2$ with quadratic form $-a_1\Nm_{\wf_2/\Q}$.
The inner integral of $W(\hb)$ over $[G]$ becomes
\begin{align}\label{inner whittaker one}
&\int_{[G]} \theta_{V} ( g d(\nu(\hb)),n(x)\hb,\varphi) \theta ( g d(\nu(\hb)), h \lambda,\Xi) dg \notag\\
=& \int_{[G]} \sum_{\nu \in \wf_2} \lb \omega_0(g)
\Fc (L( n(x)\hb )\varphi) \rb \lb (0,0), \nu  \rb 
\times \sum_{\beta\in F} \lb \omega \lb g \rb L\lb h\lambda \rb \Xi \rb (\beta) dg\\
+&\int_{[G]} \sum_{\nu \in \wf_2} \label{inner whittaker two}
\sum_{\substack{\gamma\in N(\Q) \backslash \SL_2(\Q)}} 
\lb \omega_0(g) \Fc (L( n(x)\hb )\varphi) \rb \lb (0,1)\gamma g, \nu  \rb
\times \sum_{\beta\in F} \lb \omega \lb g \rb L\lb h\lambda \rb \Xi \rb (\beta) dg.
\end{align}
Suppose the Schwartz function $(\varphi,\Xi)$ is invariant under the Weil representation on $K=\SL_2(\widehat{\Z})\SO_2(\R)$. 
Then the integrand
$\theta_{V} ( g d(\nu(\hb)),n(x)\hb,\varphi)  \theta ( g d(\nu(\hb)), h \lambda,\Xi)$ is right invariant under $d(\nu(\hb)) \SL_2(\widehat{\Z}) \SO_2(\R) d(\nu(\hb))^{-1}$.
We normalize the Haar measures such that $\vol (K) = \vol (\widehat{\Oc}_{\wf_2}^\times) = 1$.
Using the modified Iwasawa decomposition
\begin{align*}
\SL_2(\A) = N(\A)M(\A^\times) d(\nu(\hb)) K d(\nu(\hb))^{-1} ,\quad   
g =  n(b)m(a)d(\nu(\hb)) k d(\nu(\hb))^{-1}  
\end{align*}
and Lemma \ref{weil mixed model}, the integral \eqref{inner whittaker two} equals  
\begin{align*}
&\int_{N(\Q)\backslash G(\A)} \sum_{\nu \in \wf_2} \omega_0(g)
\Fc (L( n(x)\hb )\varphi) \lb (0,1) g, \nu  \rb
\times \sum_{\beta\in F} \omega \lb g \rb L\lb h\lambda \rb \Xi(\beta) dg\\
=& \int_{N(\Q)\backslash G(\A)} \sum_{\nu \in \wf_2} 
\Fc ( L( n(x)\hb) \omega (d(\nu(\hb))^{-1}g d(\nu(\hb)) \varphi) \lb (0,1), \nu  \rb \\
&\quad \times \sum_{\beta\in F} L\lb h\lambda \rb \omega (d(\nu(\hb))^{-1}g d(\nu(\hb)) \Xi(\beta) dg\\
=& |\nu(\hb)| \int_{\Q\backslash\A } \int_{ \A_f^\times \times \R^+} \sum_{\nu \in \wf_2} 
\Fc ( \omega(n(b)m(a)) L( n(x)\hb) \varphi) \lb (0,1), \nu  \rb \\
&\quad \times \sum_{\beta\in F} \omega(n(b)m(a)) L\lb h\lambda \rb  \Xi(\beta) db \frac{d^\times a}{|a|^2}\\
\end{align*}
by \eqref{weil left} and \eqref{weil right}. Applying Lemma \ref{weil mixed model} again, we obtain 
\begin{align*}
& |\nu(\hb)| \int_{\Q\backslash\A } \int_{\A_f^\times \times \R^+}  
\sum_{\nu \in \wf_2} \psi (-b a_1\Nm (\nu)) \chi_{\wf_2}(a) |a|   
\Fc (L( n(x)\hb )\varphi) \lb (0,a^{-1}), a\nu \rb \\
&\qquad\times \sum_{\beta\in F} \psi (b a_2\Nm (\beta)) \chi_F(a) |a|  L\lb h\lambda \rb \Xi(a\beta) db \frac{d^\times a}{|a|^2}\\
=& |\nu(\hb)| \int_{\A_f^\times \times \R^+} \chi_{\wf_2}(a) \chi_F(a) 
\sum_{\nu \in \wf_2}  \Fc (L( n(x)\hb )\varphi) \lb (0,a^{-1}), a\nu \rb 
\sum_{\substack{\beta\in F \\ \Nm (\beta)=\frac{a_1}{a_2} \Nm(\nu)}}  L\lb h\lambda \rb \Xi(a\beta)  d^\times a .
\end{align*}
The integral of the first summation over $\wf_2\backslash \A_{\wf_2}$ is
\begin{align*}
&\int_{\wf_2\backslash \A_{\wf_2}}  
\sum_{\nu \in \wf_2}  \Fc (L( n(x)\hb )\varphi) \lb (0,a^{-1}), a\nu \rb  \psi_{\wf_2} (-x) dx \\
=&  \int_{\wf_2\backslash \A_{\wf_2}}  
\sum_{\nu \in \wf_2} \e (a_1 \tr (x\nu)) \Fc (L( \hb )\varphi) \lb (0,a^{-1}), a\nu \rb  \psi_{\wf_2} (- x) dx \\
=& \Fc (L( \hb )\varphi) \lb (0,a^{-1}), a/a_1 \rb,
\end{align*}
where we used Lemma \ref{weil mixed model} in the first equality and only $\nu=1/a_1$ survives in the second equality.
Hence the contribution of \eqref{inner whittaker two} in the Whittaker function is  
\begin{align*}
&|\nu (\hb)| \int_{F^1\backslash \A_F^1} \int_{\A_f^\times \times \R^+}  \chi_{\wf_2}(a)
\Fc (L (\hb) \varphi)\lb (0,a^{-1}),a/a_1 \rb
\times \chi_F(a) \sum_{\beta\in \beta_0 F^1}   (L(h\lambda) \Xi)(a\beta) d^\times a  \rho (h\lambda) dh\\
=& |\nu (\hb)| \int_{\A_F^1} \int_{\A_f^\times \times \R^+} \chi_{\wf_2}(a) \chi_F(a) \Fc (L (\hb) \varphi)\lb (0,a^{-1}),a/a_1 \rb
\times (L(h\lambda) \Xi)(a \beta_0)  d^\times a  \rho (h\lambda) dh
\end{align*}
since $\rho$ is trivial on $F^\times$, where $\beta_0\in F$ satisfies $\Nm\beta_0 = \frac{1}{a_1a_2}$. 
If such $\beta_0$ doesn't exist, the Whittaker function becomes zero.
By a similar argument, the integral \eqref{inner whittaker one} contributes 0 to the Whittaker function. 
Putting these together, we get
\begin{align}\label{whittaer doi}
W(\hb) = |\nu (\hb)|
\int_{\A_{F^1}} \int_{\A_f^\times \times \R^+} 
\chi_{\wf_2}(a) \chi_F(a) \Fc (L (\hb) \varphi)\lb (0,a^{-1}),a/a_1 \rb
\times (L(h\lambda) \Xi)(a \beta_0)  d^\times a  \rho (h\lambda) dh.
\end{align}
If the level of $(\varphi,\Xi)$ is $\Gamma$, an open compact subgroup of $\SL_2(\widehat{\Z})$, the $(\varphi,\Xi)$ in previous formula needs to be replaced by $ \sum\limits_{\gamma \in \Gamma \backslash \SL_2(\widehat{\Z})} \lb \omega(\gamma) \varphi , \omega(\gamma) \Xi \rb$.

\subsubsection{Values at $T^1$ of Whittaker functions}
Suppose $\alpha\in \A_{\wf_2}^\times$ and $\lambda\in \A_F^\times$ satisfy $\Nm\ \alpha=\Nm\ \lambda^{-1}$. 
Let $\hb = i\lb t(\alpha)\rb = \lb t(\alpha), \Nm\ \alpha \rb$. By definition,
\begin{align*}
&\Fc (L (\hb) \varphi)\lb \eta_1,\eta_2,\nu \rb
=|\Nm \alpha|_{\A} \Fc (\varphi)\lb \eta_1 \Nm\alpha,\eta_2,\nu \alpha^{\prime}\rb.
\end{align*}
Substituting into \eqref{whittaer doi}, we obtain
\begin{align}\label{whittaer doi t(alpha)}
\begin{split}
W\lb \hb \rb
=& |\Nm \alpha|_{\A}^{1/2}  \int_{\A_{F^1}} \int_{\A_f^\times \times \R^+} 
\chi_{\wf_2}(a) \chi_F(a) \Fc (\varphi)\lb (0,a^{-1},a\alpha^{\prime} /a_1 \rb 
 \times \Xi(a \beta_0 h^{-1} \lambda^{-1})  d^\times a \rho(h\lambda) dh.
\end{split}
\end{align}
Let $\mfp \subset M_{32}$ be a prime ideal lying over a rational prime $p\in\Q$ unramifed in $M_{32}$, and denote by $\bfrob_\mfp = (1,0,0)^a (1,0,1)^b\in\Gal(M_{32}/\Q)$ its image under the Artin map.

\begin{proposition}\label{value T1 doi}
For a rational prime $p\in\Q$ unramifed in $M_{32}$, let $a_{1,p} = b_{1,p} =1$ and the Schwartz function 
\begin{align*}
(\varphi_{0,p},\Xi_{0,p}) = 
\cha (\Z_p^2 \oplus \mathcal{O}_{\wf_2,p}) \otimes \cha (\mathcal{O}_{F,p}).
\end{align*}
Then, the integral 
\begin{align*}
&\widetilde{\phi}_p(\alpha):=\int_{F_p^1} \int_{\Q_p^\times} \chi_{\wf_2}(a) \chi_F(a) 
\Fc (\varphi_{0,p})\lb (0,a^{-1}),a\alpha^{\prime} \rb
\times \Xi_{0,p}(a h^{-1} \lambda^{-1}) d^\times a \rho(h\lambda) dh\\
=&\begin{cases}
\zeta_8^{b(v_{\fp_1}(\alpha)+v_{\fp_2}(\alpha))} (v_{\mathfrak{p}_1}(\alpha)+1)(v_{\mathfrak{p}_2}(\alpha)+1)  
& \text { if } 4 \mid a \text{ and } 2\mid b \\
\zeta_8^{b(v_{\fp_1}(\alpha)+v_{\fp_2}(\alpha))} \frac{(-1)^{v_{\mathfrak{p}_1}(\alpha)}+1}{2} \frac{(-1)^{v_{\mathfrak{p}_2}(\alpha)}+1}{2} 
& \text { if } 2 \mid a,\ 4 \nmid a \text{ and } 2\mid b \\
\frac{(-1)^{v_{\fp}(\alpha)}+1}{2} 
& \text { if } 2 \nmid a \text{ and } 2\mid b\\
\zeta_8^{b(v_{\fp_1}(\alpha)+v_{\fp_2}(\alpha))} \frac{(-1)^{v_{\mathfrak{p}_1}(\alpha)}+1}{2} \frac{(-1)^{v_{\mathfrak{p}_2}(\alpha)}+1}{2} 
& \text { if } 4 \mid a \text{ and } 2\nmid b\\
\zeta_8^{-b(v_{\fp_1}(\alpha)+v_{\fp_2}(\alpha))} \frac{(-1)^{v_{\mathfrak{p}_1}(\alpha)}+1}{2} \frac{(-1)^{v_{\mathfrak{p}_2}(\alpha)}+1}{2} 
& \text { if } 2 \mid a,\ 4\nmid a \text{ and } 2\nmid b\\
i^{(a+b) v_\fp(\alpha)} (v_{\fp}(\alpha)+1)
& \text { if } 2 \nmid a \text{ and } 2\nmid b
\end{cases}.
\end{align*}
Here we normalize the Haar measures such that $\vol (\Z_p^\times) = 1$, and if $p$ is inert in $F$, $\vol (F_p^1)=1$.
\end{proposition}

\begin{remark}
Note that if such $\lambda$ satisfying $\Nm\ \alpha=\Nm\ \lambda^{-1}$ doesn't exist, the value of the previous Whittaker function at $t(\alpha)$ is zero. 
\end{remark}

\begin{proof}
Depending on the splitting behavior of $p$ in $F$ and $\wf_2$, there are four cases to consider.

(i) If $p$ is both split in $\wf_2$ and $F$, then $\bfrob_\mfp = (1,0,0)^{a} (1,0,1)^{b}$ with $a,b$ even. 
Suppose $\alpha=(p^i u,p^j v)$, $i,j\in \Z,u,v \in \Z_p^\times$, $v_p(\alpha) :=\min (i,j) = i$ and $\lambda=\alpha^{-1}\in\mathcal{O}_{F,p}$. 
Let $h=(p^k w,p^{-k}w^{-1})$, $k\in \Z,w\in \Z_p^\times$. 
If $4 \mid a$, $\rho(\lambda)= \zeta_8^{b(i+j)}$, and 
\begin{align*}
\widetilde{\phi}_p(\alpha) 
&=\sum_{\ell=-i}^{0} \int_{-\ell -j \leq k \leq \ell+i} dh
= (i+1)(j+1)
\end{align*}
when $i,j\geq 0$ and 0 otherwise. 
When $2 \mid a, 4\nmid a$, we have 
\begin{align*}
\widetilde{\phi}_p(\alpha) 
&= \sum_{\ell=-i}^{0} \sum_{k=-\ell -j}^{\ell+i} (-1)^k 
 = \frac{(-1)^i+1}{2} \frac{1+(-1)^j}{2}
\end{align*}
when $i,j\geq 0$ and 0 otherwise. 
It's non-zero when $i$ and $j$ are even, and in this case, $\rho(\lambda)= \zeta_8^{b(i+j)}.$

(ii) If $p = \fp$ is inert in $\wf_2$ and $p=\mathfrak{q}_1\mathfrak{q}_2$ is split in $F$, then $\bfrob_\mfp=(1,0,0)^a (1,0,1)^b$ with $a$ odd and $b$ even.
In this case, $\frob_\mathfrak{p}=(2a+b,b,0)$ and $\frob_{\mathfrak{q}_1}\frob_{\mathfrak{q}_2}^{-1}= (a,-a,0)$.
Let $\alpha=p^{j} u $, $j\in \Z, u\in \Oc_{\wf_2,\fp}^\times$, and $\lambda=(p^{-j} v,p^{-j} w) \in \Oc_{F,p}$, $v,w\in\Z_P^\times$, such that $\Nm \alpha = \Nm \lambda^{-1}$. Then $\rho(\lambda) = i^{j(a+b)}$, and 
\begin{align*}
\widetilde{\phi}_p(\alpha) =
\sum_{\ell=-j}^{0} (-1)^\ell \sum_{k=-\ell-j}^{\ell+j} (\pm i)^k = \frac{1+(-1)^j}{2} i^j
\end{align*}
when $j\geq 0$ and 0 otherwise. 
Put these together, we obtain $\frac{1+(-1)^{j}}{2}$.

(iii) If $p$ is split in $\wf_2$ and inert in $F$, then $\bfrob_\mfp=(1,0,0)^a (1,0,1)^b$ with $a$ even and $b$ odd. 
Let $\alpha= (p^i u, p^j v)$, $i,j\in\Z,u,v \in \Z_p^\times$. Then $\Nm\ \alpha \in \Nm\ \A_F^\times$ if and only if $i$ and $j$ have the same parity.
Suppose $\lambda=p^{\frac{i+j}{-2}}w \in F_p$ such that $\Nm\ \alpha = \Nm\ \lambda ^{-1}$, and $v_p(\alpha)=i$. Then
\begin{align*}
\widetilde{\phi}_p(\alpha) 
&=\sum_{\ell=-i}^{0} (-1)^\ell = \frac{(-1)^i+1}{2}
\end{align*}
when $j\geq i\geq 0$ and 0 otherwise.
In this case, 
$$\rho(\lambda) 
= \rho (a+b,a+b,0) ^{\frac{i+j}{-2}} 
=\begin{cases}
i^{b\frac{i+j}{2}}, & 4\mid a\\
i^{b\frac{i+j}{-2}}, & 4\nmid a 
\end{cases}
.$$

(iv) If $p$ is inert in both $\wf_2$ and $F$, then $\bfrob_\mfp = (1,0,0)^a (1,0,1)^b$ with $a,b$ odd.
Let $\alpha= p^i u$, $i\in\Z,u\in \Oc_{\wf_2,p}^\times$, and take $\lambda=\alpha^{-1}\in\mathcal{O}_{F,p}$. 
Then $\rho(\lambda) 
= i^{(a+b) i}$, and
\begin{align*}
\widetilde{\phi}_p(\alpha) 
&
= i+1
\end{align*}
when $ i\geq 0$ and 0 otherwise. 
\end{proof}

\section{Twisted Siegel-Weil formulas}
In this section, we prove the twisted Siegel-Weil formulas by showing that the base change of Jacquet-Langlands correspondence by Doi-Naganuma lift is an isomorphism, which implies that the twisted theta integral over $\wk_4$ against $\chi$ is realized as the Doi-Naganuama lift of Hecke's integral over $F$ against $\rho$.

\subsection{Base change of Jacquet-Langlands correspondence}

For the character $\chi$ on $\wk_4^\times \backslash \A_{\wk_4}^\times$ defined in Lemma \ref{wk4 character}, we denote the space
\begin{align*}
\Theta(\chi) := \{ \theta_{\wk_4,\chi}(\gb,\phi) : \phi \in \Sc(\A_{\wk_4}) \}.
\end{align*}
For each prime ideal $\fp\leq \infty$ of $\wf_2$, the subspace
\begin{align*}
\Sc (\wk_{4,\fp}, \chi):=\left\{f \in \Sc(\wk_{4,\fp}): f\left(h^{-1} x\right)=\chi(h) f(x), \forall h \in \wk_{4,\fp}^1, \forall x \in \wk_{4,\fp}\right\}
\end{align*}
is invariant under the Weil representation of $\mathrm{SL}_2(\wf_{2,\mathfrak{p}})$, and extended to a representation $\pi_{\chi,\mathfrak{p}}$ of $\GL_2(\wf_{2,\fp})$ by Jacquet and Langlands \cite[Proposition 1.5]{JL70}.
Moreover, $\pi_{\chi,\fp}$ is admissible and irreducible of central character $\chi_{\wf_4 / \wf_2} \chi|_{\wf_2^\times}$ \cite[Theorem 4.6]{JL70}.
Using global theta correspondence, Harris and Kudla showed the following local-global compatibility.

\begin{proposition}
The restricted tensor product $\pi_\chi$ of the local representations $\pi_{\chi,\fp}$ is isomorphic to the automorphic cuspidal representation $\Theta (\chi)$ of $\GL_2(\A_{\wf_2})$.
\end{proposition}
\begin{proof}
This is a special case of \cite[Section 13]{HK91}, see also \cite[Theorem 2.3.3]{popa06}.
\end{proof}

\begin{remark}
In our case, $\chi$ is unramified at all places. So the theta function associated to the Schwartz function defined in Proposition \ref{value T1 K4} is the Whittaker newform for the representation $\pi_\chi$ \cite{popa06}.
\end{remark}

Similarly, for the character $\rho$ on $F^\times \backslash \A_{F}^\times$ defined in Lemma \ref{matching central characters}, we have
\begin{align*}
\pi_\rho \cong \{ \theta_{\rho}(g,\Xi) : \Xi \in \Sc(\A_{F}) \}
\end{align*}
is an automorphic cuspidal representation of $\GL_2(\A)$.
Let the space of Doi-Naganuma lift of $\pi_\rho$ be
\begin{align*}
\Theta(\pi_\rho) := \{ I(\hb,\varphi,\Xi,\rho) : \varphi\in \Sc(V)(\A),\Xi\in \Sc(\A_F) \}.
\end{align*}

\begin{theorem}[Twisted Siegel-Weil formulas]\label{twisted siegel weil}
Let $\chi$ and $\rho$ be characters as above. Then
\begin{align}
    \pi_\chi \cong \Theta(\pi_\rho).
\end{align}
Consequently, for any Schwartz function $\phi\in\Sc (W(\A_{\wf_2}))$, there exists finitely many Schwartz functions $(\varphi,\Xi) = \sum_i (\varphi_i,\Xi_i), \varphi_i\in \Sc (V(\A)),\Xi_i \in \Sc (U(\A))$, such that 
\begin{align}
    \theta_{\wk_4,\chi} (\gb,\phi) = I (\hb,\varphi,\Xi,\rho),\quad \gb=\hb \in \GL_2(\A_{\wf_2}).
\end{align}
Moreover, all almost all places, $(\varphi,\Xi)$ is the characteristic function of the maximal integral lattice.
\end{theorem}
\begin{proof}
Let $\phi_0 \in \Sc(\A_{\wk_4})$ be as in Proposition  \ref{value T1 K4} and $(\varphi_0,\Xi_0)$ as in Proposition  \ref{value T1 doi}. 
Then at any place over a rational prime $p$ unramified in $M_{32}$, the local Whittaker functions of $\theta_{\wk_4,\chi} (\gb,\phi_0)$ and $I(\hb,\varphi_0,\Xi_0,\rho)$ are the same by Lemma \ref{matching unramified}. 
Applying strong multiplicity one for automorphic representations of $\GL_2(\A_{\wf_2})$ \cite[Theorem 4.10]{RT11}, $\theta_{\wk_4,\chi} (\gb,\phi_0)$ and $I(\hb,\varphi_0,\Xi_0,\rho)$ generate the same automorphic representation. Therefore, $\pi_\chi \cong \Theta (\chi) = \Theta(\pi_\rho)$.
\end{proof}

\begin{remark}
As seen in the previous proof, if we replace the quadratic form by a scalar multiple, Theorem \ref{twisted siegel weil} still holds.
\end{remark}

\subsection{Matching unramified parts}

\begin{lemma}\label{matching unramified}
For any rational prime $p$ unramified in $M_{32}$, we have
\begin{align*}
\prod_{\mathfrak{p}\mid p} W_\fp(\gb,\phi_0) = W_p(\hb,\varphi_0,\Xi_0,\rho),\quad \gb=\hb \in \GL_2(\wf_{2,p}).
\end{align*}
\end{lemma}
\begin{proof}
It's easy to verify that both sides have the same behavior under the action of $n(b)\in N(\A)$, namely, multiplication by $\e (\tr (b))$. 
Since they also have the same central character by Lemma \ref{matching central characters},
it suffices to show that $W(\gb) = W(\hb)$ when $\gb=\hb=t(\alpha)$ for any $\alpha\in \wf_{2,p}^\times$.

Let $\mfp \subset M_{32}$ be a prime ideal lying over $p$, and denote by $\bfrob_\mfp = (1,0,0)^a (1,0,1)^b\in\Gal(M_{32}/\Q)$ its image under the Artin map. 
We prove that the Whittaker functions of $\theta_{\wk_4,\chi}(\gb,\phi_0)$ and $I(\gb,\varphi_0,\Xi_0)$ are the same when $b$ is odd, and the even case can be proved similarly.

Case $4\mid a$: $p=\fp_1\fp_2$ is split in $\wf_2$. In this case, $\frob_{\fp_1} = (1,0,1)^b$ and $\fp_1$ is split in $\wk_4$, $\frob_{\fp_2} = (1,0,0)^{2} (1,0,1)^{5b-6}$ and $\fp_2$ is inert in $\wk_4$. We have
\begin{align*}
RHS &= \zeta_8^{b(v_{\fp_1}(\alpha)+v_{\fp_2}(\alpha))} \frac{(-1)^{v_{\mathfrak{p}_1}(\alpha)}+1}{2} \frac{(-1)^{v_{\mathfrak{p}_2}(\alpha)}+1}{2}\\ 
 &= \zeta_8^{b v_{\fp_1}(\alpha)} \frac{(-1)^{v_{\mathfrak{p}_1}(\alpha)}+1}{2} \times 
\zeta_8^{-(5b-6) v_{\fp_2}(\alpha)} \frac{(-1)^{v_{\mathfrak{p}_2}(\alpha)}+1}{2}
=LHS
\end{align*}
by Proposition \ref{value T1 K4} and Proposition \ref{value T1 doi}.

Case $4\nmid a,\ 2\mid a$: $p=\fp_1\fp_2$ is split in $\wf_2$; $\frob_{\fp_1} = (1,0,0)^2 (1,0,1)^b$ and $\fp_1$ is inert in $\wk_4$;  $\frob_{\fp_2} = (1,0,1)^{5b-6}$ and $\fp_2$ is split in $\wk_4$. Then
\begin{align*}
RHS &= \zeta_8^{-b(v_{\fp_1}(\alpha)+v_{\fp_2}(\alpha))} \frac{(-1)^{v_{\mathfrak{p}_1}(\alpha)}+1}{2} \frac{(-1)^{v_{\mathfrak{p}_2}(\alpha)}+1}{2} \\ 
 &= \zeta_8^{(5b-6) v_{\fp_1}(\alpha)} \frac{(-1)^{v_{\mathfrak{p}_1}(\alpha)}+1}{2} \times 
\zeta_8^{-b v_{\fp_2}(\alpha)} \frac{(-1)^{v_{\mathfrak{p}_2}(\alpha)}+1}{2} 
= LHS.
\end{align*}

Case $2\nmid a$: $p=\fp$ is inert in $\wf_2$, $\frob_{\fp} = (a+b,a+b,0)$ and $\fp$ is split in $\wk_4$. We obtain
\begin{align*}
RHS &= i^{(a+b) v_\fp(\alpha)}  ( v_{\fp}(\alpha)) + 1 ) = LHS. 
\end{align*}
This finishes the proof.
\end{proof}

Let $I_{\mathcal{Q}} \subset \Gal (M_8/\Q)$ be the possible inertia subgroup of a prime ideal $\mathcal{Q}\subset M_{8}$ lying over a prime ideal $\fp \subset \wf_2$ and a rational prime $p$. 
The possible ramification of $p$, resp. $\fp$, in $F,F_2,\wf_2$, resp. $\wk_4$ is as follows, where Y or N represents $p$, resp. $\fp$, being ramified or unramified in the corresponding field. 

\begin{center}
\begin{tabular}{|c|c|c|c|c|}\hline
$I_{\mathcal{Q}}$             & $F$ & $F_2$ & $\wf_2$ & $\wk_4$\\ \hline
$\langle (1,0,0) \rangle$     & N & Y     & Y     &  Y   \\ \hline
$\langle (2,0,0) \rangle$     & N & N     & N     &  Y   \\ \hline
$\langle (0,0,1) \rangle$     & Y & N     & Y     &  N   \\ \hline
$\langle (1,0,1) \rangle$     & Y & Y     & N     &  N   \\ \hline
$\langle (2,0,1) \rangle$     & Y & N     & Y     &  N   \\ \hline
$\langle (3,0,1) \rangle$     & Y & Y     & N     &  Y   \\ \hline
\end{tabular}
\end{center}

At ramified places, the fundamental invariant vectors are used to match Hilbert Eisenstein series by Doi-Naganuma lift in \cite{LZ25}. 
However, this natural candidate doesn't match the twisted theta integrals in our setting comparing Proposition \ref{value T1 K4} with the following proposition.

\begin{proposition}\label{value T1 ramified}
For an odd prime $p=\fp^2$ ramified in $\wf_2$, set $a_{1,p}=a_{2,p}=1$ if $p\mid d_F$, and $a_1=p^{-1}$ and $a_2=p^n$ such that $\frac{1}{p^{n-1}}$ is the norm of $\beta_0\in F_p$ otherwise.
Let the Schwartz function 
\begin{align*}
(\varphi_p,\Xi_p) = 
\begin{cases}
\cha ( \Z_p^2 ) \otimes \mathfrak{u}_{2,p}, & p\mid d_{\wf_2},p\mid d_F,\\
\mathfrak{u}_{3,p} \otimes \cha (\Z_p \times p^{-n}\Z_p), & p\mid d_{\wf_2},p\nmid d_F.
\end{cases}
\end{align*}
Then, the integral
\begin{align*}
\widetilde{\phi}_p(\alpha)
:=&\int_{F_p^1} \int_{\Q_p^\times} \chi_{\wf_2}(a) \chi_F(a) 
\Fc (\varphi_p)\lb (0,a^{-1}),a\alpha^{\prime}/a_1 \rb
\times \Xi_p(a \beta_0 h^{-1} \lambda^{-1}) d^\times a \rho(h) dh\\
=&\begin{cases}
\frac{1-(-1)^{i}}{4} \cha_{\R_{\leq 0}}(-i-1) + \sum_{\ell= \lceil\frac{-i}{2} \rceil}^{0} 1  .
& p\mid d_{\wf_2},p\mid d_F,\\
\cha_{\R_{\geq 0}}(v_\mathfrak{p}(\alpha)+1) \times (v_\mathfrak{p}(\alpha)+2) & p\mid d_{\wf_2},p\nmid d_F.
\end{cases}
\end{align*}
Here we normalize the Haar measures such that $\vol (\Z_p^\times) = \vol(F_p^1) = 1$. 
\end{proposition}
\begin{proof}
(1) When $p$ is ramified in both $\wf_2$ and $F$, denote $\varpi_p$ as the uniformizer of $\wf_{2,p}=F_p$.
Let $\alpha=\varpi_p^i u$ and $\lambda=\varpi_p^{-i} u^{-1}$.  
Substituting Equation \eqref{invariant vector 2}, we obtain
\begin{align*} 
\widetilde{\phi}_p(\alpha) 
= \frac{1-(-1)^{i}}{4} \cha_{\R_{\leq 0}}(-i-1) + \sum_{\ell= \lceil\frac{-i}{2} \rceil}^{0} 1  .
\end{align*}

(2) If $p$ is ramified in $\wf_2$ and unramified in $F$, $p =\fq_1\fq_2$ is split in $F$ and $\fq_i$ is ramified in $M_{32}$, $i=1,2$. 
Let $\lambda = (1,\Nm \alpha^{-1})$ and $\beta_0 = (1,p^{1-n})$.
Applying Equation \eqref{invariant vector 3}, we have
\begin{align*}
\widetilde{\phi}_p(\alpha) = v_\fp(\alpha) +2
\end{align*}
when $v_\fp(\alpha) \geq  -1$ and 0 otherwise. 
\end{proof}

\subsection{Matching archimedean parts}
Note that $\rho$ is trivial on infinite parts since $\sigma^2$ is complex conjugation and $\rho(\sigma^2)=1$ by Lemma \ref{matching central characters}.
Let Schwartz functions
\begin{align}\label{weight pm 1}
\begin{split}
\varphi_{\infty}^{(1,1)}\left(a, b, \nu, \nu^{\prime}\right) &:= -i \left(a-b+i\left(\nu+\nu^{\prime}\right)\right) 
e^{-\pi\left(a^2+b^2+\nu^2+\left(\nu^{\prime}\right)^2\right)},\\
\varphi_{\infty}^{( \pm 1, \mp 1)}\left(a, b, \nu_1, \nu_2\right) &:=\left( \pm i(a+b)+\left(\nu_1-\nu_2\right)\right) 
e^{-\pi\left(a^2+b^2+\nu_1^2+\nu_2^2\right)} \in \mathcal{S}\left( \mathbb{R}^{2,2},\det\right) .
\end{split} 
\end{align} 
Then, by \cite[(4.19)]{BLY22},
\begin{align*}
\Fc\left(\varphi_{\infty}^{\left(k, k^{\prime}\right)}\right)((0, r), \nu)
=\varphi_{0, \infty}^{\left(k, k^{\prime}\right)}(i r, 0, \nu) e^{-2 \pi r^2}, (k,k^\prime)=(1,1),(\pm 1,\mp 1).
\end{align*}
Here we record one useful integral
\begin{equation}\label{special integrals}
\begin{split}
\int_0^{\infty} r^{-1}\left( a r^{-1}+b r\right) \exp \left(-\pi a^2 r^{-2}-\pi b^2 r^2\right) d r 
&=\sgn(a)\frac{1+\sgn(ab)}{2} e^{-2 \pi a b}
.
\end{split}
\end{equation}

Now we match the archimedean part of the Whittaker functions of $\theta_{\wk_4,\chi}(\gb,\phi_f^\pm\phi_\infty^\pm)$ for any $\phi_f^\pm\in \Sc(W^\pm)(\A_{\wf_2})$ and $\phi_\infty^\pm$ the Gaussian of $W^\pm$.
Using a change of variable, we rewrite 
\begin{align}\label{rewriting wk4 chi}
\begin{split}
\theta_{\wk_4,\chi}(\gb,\phi_f^+\phi_\infty^+) 
&= \sum_{\alpha\in\wf_2^\times} W_{\wk_4,\chi,f}(t(\alpha)\gb,\phi_f^+) W_{\wk_4,\chi,\infty}(t(\alpha)\gb,\phi_\infty^+)\\
&= \sum_{\alpha\in\wf_2^\times} W_{\wk_4,\chi,f}(t(\alpha\sqrt{\widetilde{D}})\gb,\phi_f^+) W_{\wk_4,\chi,\infty}(t(\alpha)\gb,\phi_\infty^{1,-1}),\\
\theta_{\wk_4,\chi}(\gb,\phi_f^-\phi_\infty^-) 
&= \sum_{\alpha\in\wf_2^\times} W_{\wk_4,\chi,f}(t(\alpha)\gb,\phi_f^-) W_{\wk_4,\chi,\infty}(t(\alpha)\gb,\phi_\infty^-)\\
&= \sum_{\alpha\in\wf_2^\times} W_{\wk_4,\chi,f}(t(-\varepsilon\alpha\sqrt{\widetilde{D}})\gb,\phi_f^-) W_{\wk_4,\chi,\infty}(t(\alpha)\gb,\phi_\infty^{-1,1}),
\end{split}
\end{align}
where 
\begin{align*}
W_{\wk_4,\chi,\infty}(t(\alpha),\phi_\infty^{1,-1}) &= \frac{1+\sgn(\alpha)}{2} \frac{1-\sgn(\alpha')}{2} \exp (-2\pi (\alpha - \alpha')),\\
W_{\wk_4,\chi,\infty}(t(\alpha),\phi_\infty^{-1,1}) &= \frac{1-\sgn(\alpha)}{2} \frac{1+\sgn(\alpha')}{2} \exp (-2\pi (-\alpha + \alpha')).
\end{align*}

\begin{proposition}\label{match weight}
Let $\Xi^\pm_\infty (x_1,x_2) = \left(x_1 \pm x_2\right) e^{-\pi\left(x_1^2+x_2^2\right)} \in \mathcal{S}\left(\mathbb{R}^2\right)$ and $a_{1,\infty} = a_{2,\infty} =1$.  
Then, when $\gb=\hb \in \GL_2(\wf_{2,\infty})$,
\begin{align*}
W_{\wk_4,\chi,\infty}(\gb,\phi_\infty^{1,1}) &= \frac{1}{2} W_\infty(\hb,\varphi_\infty^{( 1, 1)},\Xi_\infty^+,\rho),\\
W_{\wk_4,\chi,\infty}(\gb,\phi_\infty^{\pm 1,\mp 1}) &= \frac{1}{2} W_\infty(\hb,\varphi_\infty^{(\pm 1,\mp 1)},\Xi_\infty^-,\rho).\\
\end{align*} 
\end{proposition}
\begin{proof}
We prove the first case here, and the second case can be proved similarly.
By Proposition \ref{value T1 K4}, it suffices to show that 
\begin{align*}
W_\infty((t(\alpha),t(\alpha')),\varphi_\infty^{( 1, 1)},\Xi_\infty^+) 
=2 \exp (-2\pi (\alpha + \alpha^\prime))
\end{align*}
when $\alpha ,\alpha' \in \R_{>0}$ and 0 otherwise.
Let $\lambda = (\alpha^{-1},\alpha^{\prime,-1} ) \in \R^2$.
Then
\begin{align*}
&\int_{\R^+} \int_{\R^\times} \chi_{\wf_2}(a) \chi_F(a) 
\Fc (\varphi_\infty^{1,1})\lb (0,a^{-1}),a\alpha^{\prime} \rb
\times \Xi_\infty^+(a h^{-1} \lambda^{-1}) d^\times a dh\\
=&\int_{\R^+} \int_{\R^\times} 
\left(a^{-1}+a\left(\alpha+\alpha^{\prime}\right)\right)
e^{-\pi\left(a^{-2}+(a\alpha)^2+\left(a\alpha^{\prime}\right)^2\right)}
\times \Xi_\infty^+(a h^{-1} \lambda^{-1}) d^\times a dh\\
=& 2\sgn(\alpha) \frac{1 + \operatorname{sgn}\left(\alpha \alpha^{\prime}\right)}{2}\int_{\R^+} 
\left(a^{-1}+a\left(\alpha+\alpha^{\prime}\right)\right)
e^{-\pi\left(a^{-2}+(a\alpha)^2+\left(a\alpha^{\prime}\right)^2\right)}
\times \exp (-2\pi a^2 \alpha\alpha^\prime) d^\times a\\
=&2\frac{1 + \operatorname{sgn}\left(\alpha \right)}{2} 
\frac{1 + \operatorname{sgn}\left(\alpha^{\prime}\right)}{2} \exp (-2\pi (\alpha + \alpha^\prime))
\end{align*}
by Equation \eqref{special integrals}. This finishes the proof.

\end{proof}

\section{Twisted Integral of Borcherds forms}
In this section, we recall the deformed theta integrals constructed in \cite{CL20}, see also \cite[Section 2]{BLY22}, and calculate their Whittaker functions.
Then we use the Doi-Naganuma lift of deformed theta integrals to obtain twisted CM value formulas of Borcherds forms.

\subsection{Deformed theta integrals}

Let $\rho$ the character as in Lemma \ref{matching central characters}. 
Recall $U$ is the the quadratic space $F$ over $\Q$ with quadratic form $\Nm_{F/\Q}$.
Define $H_U := \SO(U)$ and $H_U(\mathbb{Q})^+ := H_U(\mathbb{Q}) \cap H_U(\mathbb{R})^+$. 
Let $K_{\rho} \subset H_U(\widehat{\mathbb{Z}})$ be the open compact subgroup fixing $\rho$. 
The intersection $H_U(\mathbb{Q})^+\cap K_{\rho}$ is the cyclic group 
$$\Gamma_{\rho}=\left\langle\varepsilon_\rho \right\rangle,\quad \varepsilon_\rho >1>\varepsilon_\rho^{\prime}>0$$
of totally positive units in $\mathcal{O}_{F}$. 
Fix $C \subset H_U(\widehat{\mathbb{Q}})$ to be a finite subset of elements representing $H_U(\mathbb{Q})^{+} \backslash H_U(\widehat{\mathbb{Q}}) / K_{\rho}$. 
Then we have
$$H_U(\mathbb{A})=\coprod_{\xi \in C} H_U(\mathbb{Q}) H_U(\mathbb{R})^{+} K_{\rho} \xi.$$
So every element $h= ( h_f, h_{\infty} )\in H_U(\mathbb{A})$ has a decomposition $h=\left(\alpha k \xi, \alpha t\right)$, where $\alpha \in H_U(\mathbb{Q}), t \in H_U(\mathbb{R})^{+}, k \in K_{\rho}, \xi \in C$. 
This gives us the identification
\begin{equation}\label{sta so u}
H_U(\mathbb{Q}) \backslash H_U(\mathbb{A}) / K_{\rho} \cong \coprod_{\xi \in C} \Gamma_{\rho} \backslash H_U(\mathbb{R})^{+} \xi
: \left(\alpha k \xi, \alpha t\right) \mapsto \Gamma_\rho t \xi.
\end{equation}
To consider deformation of theta integrals $\theta_{\rho} (g,\Xi)$, we define the function 
\begin{align}\label{Deformed character}
\lg &: H_U(\mathbb{Q}) \backslash H_U(\mathbb{A}) / K_{\rho} \rightarrow (0,1),\quad 
\lg \left(\left(\alpha \xi,( \alpha t,\alpha^\prime t^{-1})\right)\right) :=2 \log \varepsilon_\rho \left\{\log |t| / \log \varepsilon_\rho \right\},
\end{align}
where $\{a\} :=a-\lim _{\epsilon \rightarrow 0} \frac{1}{2}(\lfloor a+\epsilon\rfloor+\lfloor a-\epsilon\rfloor)$, see \cite[Section 2]{BLY22} for more details.
Denote by
\begin{align}
\wrho :=\lg \cdot \rho: F^1 \backslash \A_{F^1} \rightarrow \mathbb{C}^\times,
\end{align}
which is $K_\rho$-invariant.
The deformed theta integral constructed in \cite{CL20} is
\begin{align}\label{deformed theta integral}
\theta_{\wrho} (g,\Xi) := \int_{[H_U]} \theta ( g, h \lambda ,\Xi) \wrho(h\lambda) dh,\quad g\in\GL_2^+(\A), \Xi=\Xi_f\Xi_\infty\in\Sc(U(\A)),
\end{align}
where $\lambda\in\A_F^\times$ satisfies $\det g = \Nm\ \lambda$.

To calculate the inverse image of $\theta_{\rho} (g,\Xi)$ under the lowering operator $L$, we define the integral 
\begin{equation}
\begin{split}
\theta_{\rho,f} :=& \int_{F^1\backslash \A_{F^1,f}} \theta ( g, h\lambda,\Xi) \rho ( h\lambda ) dh.
\end{split}
\end{equation}
Recall the function $\erfc(x):=\frac{2}{\sqrt{\pi}} \int_x^{\infty} e^{-t^2} d t=1-\erf(x)$.

\begin{lemma}\label{lowering deformed} 
The deformed theta integral $\theta_{\wrho}(g,\Xi_f\Xi_\infty^\pm)$ is an automorphic form on $\GL_2(\Q)\backslash\GL_2(\A)$ of weight $\pm 1$. 
For $\tau\in\h$, denote by $\theta_{\wrho}(\tau,\Xi_f\Xi^\pm) = v^{\mp 1/2} \theta_{\wrho} (g_\tau,\Xi_f\Xi^\pm)$. 
Suppose $\alpha\in \Q^\times$ and $\lambda \in F^\times$ satisfy $\alpha =\Nm\ \lambda$. 
Let $\lambda_n = \sqrt{|\lambda/\lambda^\prime|} \varepsilon_\rho^{n}$, $n\in\Z$.
The $\alpha$-th Fourier coefficient of $\theta_{\wrho}(\tau,\Xi_f\Xi^\pm)$ is 
\begin{align*}
W^\pm(\tau,\alpha) &= W_f(\alpha) \times W_\infty^\pm(\tau,\alpha),
\end{align*}
where
\begin{align*}
W_f(\alpha) &= \int_{\A_{F^1,f}} \Xi_f \left(\alpha \lambda^{-1} h^{-1}\right) \rho \left( h \right) d h,\\
W_\infty^\pm(\tau,\alpha) &= 4 \e (\alpha u) v^{\frac{1\mp 1}{2}} \sqrt{v|\alpha|} \sgn(\lambda^\prime) ( W_1 + W_2 +W_3),\\
W_1 &= -\log \lambda_0 \times
\begin{cases}
0,& \pm \sgn(\alpha) = -\\
\exp (-2\pi v|\alpha|) / \sqrt{v|\alpha|} ,& \pm \sgn(\alpha) = +
\end{cases},\\
W_2 &= \begin{cases}
0, & \pm \sgn(\alpha) = +\\
-\left.\frac{\partial}{\partial s} K_s(2 \pi v|\alpha|)\right|_{s=1 / 2}, & \pm \sgn(\alpha) = -
\end{cases},\\
W_3 &= -\log \varepsilon_\rho \exp (2\pi v|\alpha|) \frac{1}{2\sqrt{ v |\alpha|}}  \sum_{n\in \Z} 
\erfc (\sqrt{\pi v |\alpha|} (\lambda_{n}^{-1}  \mp \sgn(\alpha) \lambda_{n} ) ),
\end{align*}
Moreover, they satisfy the differential equation 
\begin{align*}
L \theta_{\wrho} (g,\Xi_\infty^+) = \theta_{\rho} (g,\Xi_\infty^-) + \log\varepsilon_\rho \cdot \theta_{\rho,f} (g,\Xi_\infty^-).
\end{align*}
\end{lemma}

\begin{proof}
Take $g = t(\alpha)(1, g_\tau) \in \GL_2(\A)$. 
Then the Whittaker function of $\theta_{\wrho}(g,\Xi_f\Xi_\infty^\pm)$ equals
\begin{align*}
W_{\wrho}(g , \Xi_f\Xi_\infty^\pm)  
=& |\alpha|^{1/2} \chi_F(\alpha)
\int_{\A_{F^1}} \lb \omega \left(g_\tau\right) \Xi_f\Xi_\infty^\pm \rb \left(\alpha \lambda^{-1} h^{-1}\right) \wrho \left(h\lambda \right) d h.
\end{align*}
Splitting into the finite and infinite part,
\begin{align*}
W_{\wrho}(g , \Xi_f\Xi_\infty^\pm)
=& 
\int_{\A_{F^1,f}} \Xi_f \left(\alpha \lambda^{-1} h^{-1}\right) \rho \left( h \right) d h
\int_{\R^\times} \omega \left(g_\tau\right) \Xi_\infty^\pm \left(\lambda^{\prime} h^{-1}\right) \lg \left(h\lambda \right) \frac{dh}{h}.
\end{align*}
The integral at the infinite place equals
\begin{align*}
&\int_{\R^\times} \e (\alpha u) \sqrt{v} \Xi_\infty^\pm \left(\sqrt{v}\lambda^{\prime} h^{-1}\right) \lg \left(h\lambda \right) \frac{dh}{h}\\
=& 2 \log \varepsilon \e (\alpha u) v
\int_{\R^\times} (h^{-1} \lambda^\prime \pm h\lambda) \exp \lb -\pi v (h^2 \lambda^2 +h^{-2}\lambda^{\prime,2}) \rb  
\left\{ \frac{\log |h|}{\log \varepsilon} \right\} \frac{dh}{h}\\
=&  4 \e (\alpha u) v \sqrt{|\alpha|} \sgn(\lambda^\prime)
\sum_{n\in \Z} \int_{\varepsilon_\rho^n}^{\varepsilon_\rho^{n+1}} (\lambda_0^{-1} t^{-1} \pm \sgn(\alpha) \lambda_0 t ) \\
&\quad \times
\exp \lb -\pi v |\alpha| (\lambda_0^2 t^2 + \lambda_0^{-2} t^{-2}) \rb   (\log t - n \log \varepsilon_\rho ) \frac{dt}{t}.
\end{align*}
where we make the change of variable $h=t \sqrt{|\lambda^\prime/\lambda|}$.  
Rewriting $t=t\lambda_0/\lambda_0$, the summation over $n\in\Z$ decomposes as
\begin{align}\label{three integrals}
\begin{split}
&\quad \int_{-\infty}^{+\infty} (e^{-u} \pm \sgn(\alpha) e^u ) 
\exp \lb -\pi v |\alpha| (e^{2u} + e^{-2u}) \rb   u du\\
&- \log \lambda_0 \int_{0}^{\infty} ( t^{-1} \pm \sgn(\alpha) t ) 
\exp \lb -\pi v |\alpha| ( t^2 +  t^{-2}) \rb   \frac{dt}{t}\\
&- \log \varepsilon_\rho \sum_{n\in \Z} n \int_{\varepsilon_\rho^n}^{\varepsilon_\rho^{n+1}} (\lambda_0^{-1} t^{-1} \pm \sgn(\alpha) \lambda_0 t ) 
\exp \lb -\pi v |\alpha| (\lambda_0^2 t^2 + \lambda_0^{-2} t^{-2}) \rb   \frac{dt}{t}.
\end{split}
\end{align}
The first integral in \eqref{three integrals} is $-\left.\frac{\partial}{\partial s} K_s(2 \pi v|\alpha|)\right|_{s=1 / 2} $ when $ \pm \sgn(\alpha) = -$ and 0 otherwise.
Using the identity
\begin{align*}
\int_{0}^{\infty} (t^{-1} \pm t ) \exp ( -\pi v  ( t^2 + t^{-2} ) ) \frac{dt}{t} = 
\begin{cases}
0,& - \\ 
\exp (-2\pi v) / \sqrt{v} ,& + \\ 
\end{cases},
\end{align*}
the second integral in \eqref{three integrals} equals
\begin{align*}
&-\log \lambda_0 \times
\begin{cases}
0,& \pm \sgn(\alpha) = -\\
\exp (-2\pi v|\alpha|) / \sqrt{v|\alpha|} ,& \pm \sgn(\alpha) = +
\end{cases}.
\end{align*}
Applying the identity $d_t \erf (at\pm bt^{-1}) = \frac{2}{\sqrt{\pi}} (at\mp bt^{-1}) \exp (-(at+bt^{-1})^2)$, where $d_t = d\frac{\partial}{\partial_t}$, the integral
\begin{align*}
&\int_{\varepsilon_\rho^n}^{\varepsilon_\rho^{n+1}} (\lambda_0^{-1} t^{-1} \pm \sgn(\alpha) \lambda_0 t ) \exp \lb -\pi v |\alpha| (\lambda_0 t^2 + \lambda_0^{-2} t^{-2}) \rb  \frac{dt}{t}\\
=& \exp (2\pi v|\alpha|) \frac{1}{2\sqrt{ v |\alpha|}}  
\lb \erfc (\sqrt{\pi v |\alpha|} (\lambda_{n}^{-1} \mp \sgn(\alpha) \lambda_{n} ) )
- \erfc (\sqrt{\pi v |\alpha|} (\lambda_{n+1}^{-1} \mp\sgn(\alpha) \lambda_{n+1} ) ) \rb,
\end{align*}
Therefore, the third term in \eqref{three integrals} becomes
\begin{align*}
& -\log \varepsilon_\rho \exp (2\pi v|\alpha|) \frac{1}{2\sqrt{ v |\alpha|}}  \sum_{n\in \Z} 
\erfc (\sqrt{\pi v |\alpha|} (\lambda_{n}^{-1}  \mp \sgn(\alpha) \lambda_{n} ) ).
\end{align*}
Putting these together gives the $\alpha$-th Whittaker coefficient.
The last claim is proved in \cite[Theorem 2.7]{BLY22}, see also \cite[Proposition 2.2]{CL20}.
\end{proof}

\begin{remark}\label{typo bey}
There is a typo in \cite[Proposition 5.5]{CL20}, where the term $\Gamma (0, 4\pi Q(\Lambda)v)$ in $J_2(\Lambda)$ should be $W_2$ as seen in the previous proof.
\end{remark}

\subsection{Doi-Naganuma lift of deformed theta integrals}
For $\varphi \in \Sc(V(\A))$ and $\Xi \in \Sc(U(\A))$, we denote the Doi-Naganuma lift of $\theta_{\wrho} (g,\Xi)$ by
\begin{align}\label{DN deformed}
\wI (\hb,\varphi,\Xi,\rho) = \int_{[G]} \theta_V ( g d(\nu(\hb)),\hb,\varphi) 
\int_{F^1\backslash \A_{F^1}} \theta ( g d(\nu(\hb)), h\lambda,\Xi) \wrho ( h\lambda ) dh dg,
\end{align}
where $\hb \in \GL_2(\A_{\wf_2})$ and $\lambda\in \A_F^\times$ satisfy $\nu(\lambda)=\nu(\hb)$. 
This definition doesn't depend on the choice of $\lambda$. 
To consider the preimage of $I (\hb,\varphi,\Xi,\rho)$ under the lowering operator, we define the Doi-Naganuma lift of $\theta_{\rho,f}(g,\Xi)$ as
\begin{align}\label{Finite DN deformed}
I_f(\hb,\varphi,\Xi,\rho) = \int_{[G]} \theta_V ( g d(\nu(\hb)),\hb,\varphi) 
\int_{F^1\backslash \A_{F^1,f}} \theta ( g d(\nu(\hb)), h\lambda,\Xi) \rho ( h\lambda ) dh dg.
\end{align}

\begin{proposition}\label{lowering DN deformed}
For any $\varphi_f \in \Sc(V(\A_f))$ and $\Xi_f \in \Sc(U(\A_f))$, we have
\begin{align*}
&\left(L_1+L_2\right) \mathrm{RC}_r \wI (\hb,\varphi_f \varphi_\infty^{(1,1)},\Xi_f \Xi_\infty^+,\rho)\\
= & -(-4 \pi)^{-r}\left(R_1+R_2\right)^r 
I(\hb,\varphi_f (\varphi_\infty^{(1,-1)}-(-1)^r \varphi_\infty^{(-1,1)}),\Xi_f \Xi_\infty^-,\rho) \\
& -2 \log \varepsilon_\rho \widetilde{\mathrm{RC}}_r 
I_f\left(\hb,\varphi_f ( \varphi_\infty^{(1,-1)} -(-1)^r \varphi_\infty^{(-1,1)}), \Xi_f\Xi_\infty^-, \rho\right),
\end{align*}
where $\mathrm{RC}_r$ and $\widetilde{\mathrm{RC}}_r$ are differential opeartos defined in \cite[Equation (2.8)]{BLY22}.
\end{proposition}
\begin{proof}
This follows the same as the proof of Proposition 4.2 in \cite{BLY22} but without using the Siegel-Weil formulas.
\end{proof}

Now we calculate the Fourier expansion of Doi-Naganuma lift of deformed theta integrals.
Let $z=(z_1,z_2)=(x_1+iy_1,x_2+iy_2) \in \h^2$.
Define 
\begin{align*}
\wI (z,\varphi_f\varphi_\infty^{(1,1)},\Xi_f\Xi_\infty^+,\rho) 
=  \sqrt{y_1y_2}^{-1} \wI (\hb_z,\varphi_f\varphi_\infty^{(1,1)},\Xi_f\Xi_\infty^+,\rho),
\end{align*}
where $\hb_z = (h_{z_1},h_{z_2}) = (n(x_1)m(\sqrt{y_1}),n(x_2)m(\sqrt{y_2}))$.

\begin{proposition}\label{FE deformed}
Let $\alpha\in \wf_2^\times$ and $a_1=a_2=1$.
Suppose that there exists $\lambda\in F$ such that $\Nm\ \alpha = \Nm\ \lambda^{-1}$. 
Denote by $\lambda_n = \sqrt{|\lambda^\prime/\lambda|}\varepsilon_\rho^{n}$, $n\in\Z$. 
Then the $\alpha$-th Fourier coefficient of $\wI (z,\varphi_f\varphi_\infty^{(1,1)},\Xi_f\Xi_\infty^+,\rho)$ is 
$\widetilde{W}(z,\alpha) 
=\wW_f(\alpha)  \wW_\infty(z,\alpha) $, where
\begin{align}
\wW_f(\alpha)  &=  \int_{\A_{F^1,f}} \int_{\A_f^\times} \label{finit whittaker deformed}
\chi_{\wf_2}(a) \chi_F(a) \Fc (\varphi_f)\lb (0,a^{-1},a\alpha^{\prime} \rb
\times \Xi_f (a h^{-1} \lambda^{-1})  d^\times a \rho(h) dh,\\
\wW_\infty(z,\alpha)  &= a_\alpha(z) + b_\alpha (z), \label{infinit whittaker deformed}
\end{align}
and 
\begin{align*}
a_\alpha(z) &= 
-4\log \lambda_0 \frac{1+\sgn(\alpha)}{2} \frac{1+\sgn(\alpha')}{2} \mathbf{e}\left(\alpha^\prime z_1 + \alpha z_2 \right),\\
b_\alpha(z) &=4  \sqrt{|\Nm \alpha|} \mathbf{e}\left(x_1 \alpha^\prime + x_2 \alpha \right)  \int_{\R^+} K_\alpha(a,y_1,y_2) (U_2+U_3)  d a,\\
U_2 &= \begin{cases}
0, & \sgn(\Nm\alpha) = +\\
-\left.\frac{\partial}{\partial s} K_s(2 \pi va^2|\Nm\alpha|)\right|_{s=1 / 2}, & \sgn(\Nm\alpha) = -
\end{cases},\\
U_3 &= -\log \varepsilon_\rho \exp (2\pi a^2|\Nm\alpha|) \frac{1}{2\sqrt{ a^2|\Nm\alpha|}}  \sum_{n\in \Z} 
\erfc (\sqrt{\pi a^2|\Nm\alpha|} (\lambda_{n}^{-1}  - \sgn(\Nm\alpha) \lambda_{n} ) ),\\
K_\alpha(a,y_1,y_2) &= \lb a^{-1} \sqrt{y_1y_2} + a\alpha \sqrt{y_2/y_1} + a\alpha^\prime \sqrt{y_1/y_2} \rb 
\exp (-\pi (a^{-2} y_1y_2 + a^2 \alpha^2 y_2/y_1 + a^2 \alpha^{\prime,2} y_1/y_2 ) ).
\end{align*}
Furthermore, $b_\alpha(z)$ satisfies 
\begin{align}\label{estimation rho W23}
\lim_{y\rightarrow \infty} y^{-c} \partial_{y_1}^a \partial_{y_2}^b b_\alpha (z) \mid_{y_1=y_2=y} = 0
\end{align}
for any $a,b,c\in \N_0$,
\end{proposition}

\begin{remark}\label{comparing finite whittaker}
It's clear from the expression that the finite part $\wW_f(\alpha)$ is the same the finite part of $W(t(\alpha))$ in Equation \eqref{whittaer doi t(alpha)}.
\end{remark}

\begin{proof}
Using Lemma \ref{weil mixed model}, 
the infinite part of the Whittaker function of $\wI (\hb_z,\varphi_f\varphi_\infty^{(1,1)},\Xi_f\Xi_\infty^+,\rho)$ at $t(\alpha)\hb_z$ is
\begin{align*}
&\wW (t(\alpha)\hb_z)_\infty
= \int_{\R^+} \int_{\R^\times} \Fc (L (\hb_z) \varphi_\infty^{(1,1)} ) \lb (0,a^{-1}),a\alpha^\prime \rb
\times (L \lb h \rb \Xi_\infty^+)(a \lambda^{-1})  d^\times a \wrho (h\lambda)  d^\times h\\
=& \sqrt{y_1y_2} \int_{\R^+} \int_{\R^\times} \mathbf{e}\left(x_1 \alpha^\prime + x_2 \alpha \right) 
\lb a^{-1} \sqrt{y_1y_2} + a\alpha \sqrt{y_2/y_1} + a\alpha^\prime \sqrt{y_1/y_2} \rb \\
&\times \exp (-\pi (a^{-2} y_1y_2 + a^2 \alpha^2 y_2/y_1 + a^2 \alpha^{\prime,2} y_1/y_2 ) )
\times (L \lb h \rb \Xi_\infty^+)(a \lambda^{-1})  d^\times a \wrho (h\lambda)  d^\times h.
\end{align*}
Without loss of generality, we assume $\lambda' >0$. 
The inner integral over $\R^\times$ equals
\begin{align*}
& \int_{\R^\times} \Xi_\infty^+(ha \lambda^{-1},h^{-1}a\lambda^{-1,\prime}) \wrho (h\lambda)  d^\times h \\
=& 4 \log \varepsilon_\rho  \int_{\R^+} (ha\lambda^{-1} + h^{-1}a\lambda^{-1,\prime}) \exp (-\pi (ha\lambda^{-1})^2- \pi (h^{-1}a\lambda^{-1,\prime})^2) \
\left\{ \frac{\log h}{\log \varepsilon_\rho} \right\} \frac{dh}{h}\\
=& 4 \log \varepsilon_\rho \sqrt{|\Nm \lambda^{-1}|} a 
\int_{\R^+} ( t^{-1} +\sgn (\Nm\lambda^{-1})t) \\
& \times \exp \lb -\pi a^2 |\Nm \lambda^{-1}| (t^2 + t^{-2}) \rb
\left\{ \frac{\log t}{\log \varepsilon_\rho} + \frac{1}{2}\frac{\log |\lambda^{-1,\prime}/\lambda^{-1}|}{\log \varepsilon_\rho} \right\} \frac{dt}{t},
\end{align*}
where we make the change of variable $t=h\sqrt{|\lambda^{-1}/\lambda^{-1,\prime}|}$. 
Similarly as the proof of Lemma \ref{lowering deformed}, we obtain $4  \sqrt{|\Nm \alpha|} a (U_1+U_2+U_3)$, where
\begin{align*}
U_1 &= -\log \lambda_0 \times
\begin{cases}
0,& \sgn(\Nm\alpha) = -\\
\exp (-2\pi a^2|\Nm\alpha|) / \sqrt{a^2|\Nm\alpha|} ,& \sgn(\Nm\alpha) = +
\end{cases},\\
U_2 &= \begin{cases}
0, & \sgn(\Nm\alpha) = +\\
-\left.\frac{\partial}{\partial s} K_s(2 \pi va^2|\Nm\alpha|)\right|_{s=1 / 2}, & \sgn(\Nm\alpha) = -
\end{cases},\\
U_3 &= -\log \varepsilon_\rho \exp (2\pi a^2|\Nm\alpha|) \frac{1}{2\sqrt{ a^2|\Nm\alpha|}}  
\sum_{n\in \Z} \erfc (\sqrt{\pi a^2|\Nm\alpha|} (\lambda_{n}^{-1} - \sgn(\Nm\alpha) \lambda_{n} ) ).
\end{align*}
Substituting into Whittaker functions, we get
\begin{align*}
&\wW (\pmat{\alpha}{}{}{1}\hb_z)_\infty
=  \sqrt{y_1y_2} \int_{\R^+} \mathbf{e}\left(x_1 \alpha^\prime + x_2 \alpha \right) 
\lb a^{-1} \sqrt{y_1y_2} + a\alpha \sqrt{y_2/y_1} + a\alpha^\prime \sqrt{y_1/y_2} \rb \\
&\quad \times \exp (-\pi (a^{-2} y_1y_2 + a^2 \alpha^2 y_2/y_1 + a^2 \alpha^{\prime,2} y_1/y_2 ) )
\times 4  \sqrt{|\Nm \alpha|} a (U_1+U_2+U_3)  d^\times a .
\end{align*}
When $\sgn(\Nm\alpha) = +$, the integral concerns $U_1$ is
\begin{align*}
&  -4\log \lambda_0 \sqrt{y_1y_2} \int_{\R^+} \mathbf{e}\left(x_1 \alpha^\prime + x_2 \alpha \right) 
\lb a^{-1} \sqrt{y_1y_2} + a\alpha \sqrt{y_2/y_1} + a\alpha^\prime \sqrt{y_1/y_2} \rb \\
&\exp (-\pi (a^{-2} y_1y_2 + a^2 \alpha^2 y_2/y_1 + a^2 \alpha^{\prime,2} y_1/y_2 ) ) \exp (-2\pi a^2 \alpha\alpha^\prime) d^\times a\\
=& -4\log \lambda_0 \frac{1+\sgn(\alpha)}{2} \sqrt{y_1y_2} \mathbf{e}\left(\alpha^\prime z_1 + \alpha z_2 \right)
\end{align*}
by Equation \eqref{special integrals}. Therefore, the contribution of $U_1$ in the Fourier coefficient is exactly $a_\alpha(z)$.
The other parts give $b_\alpha(z)$.
The last claim follows the same as in the proof of \cite[Proposition 4.7]{BEY21}.
\end{proof}

\subsection{Twisted integral of Borcherds forms}
Now we calculate the twisted integral of Borcherds forms defined in Section \ref{twisted CM values}, and apply it to obtain twisted CM value formulas.

\begin{theorem}\label{twisted integral botcherds forms}
Let $\chi$ and $\rho$ be the same as in Theorem \ref{twisted siegel weil}. 
For $r \in \mathbb{N}_0$ and an integral lattice $L \subset V_2(\Q)$ of level $N \in \N$, let
\begin{align*}
f=\sum_{m \in \mathbb{Q}, \mu \in L^{\vee} / L} c(m, \mu) q^m \mathfrak{e}_\mu \in M_{-2 r, \rho_L}^{!},\ \text{ and } \
P_r(x)=2^{-r} \sum_{s=0}^r\binom{r}{s}^2(x-1)^{r-s}(x+1)^s
\end{align*}
the $r$-th Legendre polynomial. Then, for any $\phi_f\in\Sc(W(\A_{\wf_2,f}))^{\SL_2(\widehat{\Z})}$, 
\begin{align*}
&(4\pi)^{-r} \int_{\SL_2(\Z)\backslash \h}^{\reg} R_\tau^r f(\tau) 
\left[ \theta_{\wk_4,\chi}(g_\tau^\Delta,\phi_f\phi_\infty^{(1,-1)}) 
- (-1)^r \theta_{\wk_4,\chi}(g_\tau^\Delta,\phi_f\phi_\infty^{(-1,1)}) \right] d\mu(\tau)\\
=& 2 \sqrt{D}^{r} \log \varepsilon_\rho \sum_{\mu\in L^\vee/L} c_\mu(f) 
- \sum_{m>0, \mu \in L^{\vee} / L} c(-m, \mu) \\
&\quad\times \sum_{\substack{\alpha \in \wf_2^+, \operatorname{Tr}(\alpha)=m \\ \exists \lambda_\alpha\in F, \text{ s.t. } \Nm \alpha = \Nm \lambda_\alpha^{-1} }} m^r P_r\left(\frac{\alpha-\alpha^{\prime}}{m}\right) 
\wW_f(\alpha;\lambda_\alpha) \log |\lambda_\alpha/\lambda_\alpha'|,
\end{align*} 
where, for the matching section $(\varphi_f,\Xi_f)$ of $\phi_f$ by Theorem \ref{twisted siegel weil}, 
\begin{align*}
c_\mu(f):= \frac{\sqrt{D}^{-r}}{|\mathrm{SL}_2(\mathbb{Z})/\Gamma(N)|} \int_{\Gamma(N) \backslash \mathbb{H}}^{\mathrm{reg}} f_\mu(\tau) (\widetilde{\mathrm{RC}}_r I_f)\left(g_\tau^{\Delta}, \varphi_f ( \varphi_\infty^{(1,-1)}-(-1)^r \varphi_\infty^{(-1,1)}),\Xi_f\Xi_\infty^-,\rho\right) d \mu(\tau),
\end{align*}
\end{theorem}
\begin{proof}
Let $\varphi_f \in \Sc (V)(\A_f)$ and $\Xi_f \in \Sc(\A_{F,f})$ be the matching sections of $\phi_f$.
Applying Proposition \ref{lowering DN deformed}, 
we have 
\begin{align*}
&(4\pi)^{-r}\int_{\SL_2(\Z)\backslash \h}^{\reg} R_\tau^r f(\tau) \left[ \theta_{\wk_4,\chi}(g_\tau^\Delta,\phi_f\phi_\infty^{(1,-1)}) 
- (-1)^r \theta_{\wk_4,\chi}(g_\tau^\Delta,\phi_f\phi_\infty^{(-1,1)}) \right] d\mu(\tau)\\
=&(-4\pi)^{-r} \int_{\SL_2(\Z)\backslash \h}^{\reg} f(\tau) \left(R_1+R_2\right)^r
\left( I(g_\tau^\Delta,\varphi_f ( \varphi_\infty^{(1,-1)}-(-1)^r \varphi_\infty^{(-1,1)}),\Xi_f\Xi_\infty^-,\rho) \right) d\mu(\tau) \\
=& - \int_{\SL_2(\Z)\backslash \h}^{\reg} f(\tau) L_\tau (\mathrm{RC}_r \wI)(g_\tau^\Delta,\varphi_f\varphi_\infty^{(1,1)},\Xi_f\Xi_\infty^+,\rho) d\mu(\tau) \\
+& 2\log\varepsilon_\rho \frac{1}{\left|\mathrm{SL}_2(\mathbb{Z})/ \Gamma(N)\right|} 
\int_{\Gamma(N) \backslash \mathbb{H}}^{\mathrm{reg}} f_\mu(\tau) 
(\widetilde{\mathrm{RC}}_r I_f)\left(g_\tau^{\Delta}, \varphi_f ( \varphi_\infty^{(1,-1)}-(-1)^r \varphi_\infty^{(-1,1)}),\Xi_f\Xi_\infty^-,\rho\right) d \mu(\tau)\\
=& 2 \sqrt{D}^{r} \log \varepsilon_\rho \sum_{\mu\in L^\vee/L} c_\mu(f) 
- \sum_{m>0, \mu \in L^{\vee} / L} c(-m, \mu) \\
&\quad\times \sum_{\substack{\alpha \in \wf_2^+, \operatorname{Tr}(\alpha)=m \\ \exists \lambda_\alpha\in F, \text{ s.t. } \Nm \alpha = \Nm \lambda_\alpha^{-1} }} m^r P_r\left(\frac{\alpha-\alpha^{\prime}}{m}\right) 
\widetilde{W}_f(\alpha;\lambda_\alpha) \log |\lambda_\alpha/\lambda_\alpha'|,
\end{align*}
by Proposition \ref{FE deformed} following the same proof in \cite[Theorem 5.1]{BEY21}.

\end{proof}

Recall the Schwartz functions $\phi_f^\pm\in\Sc (W^\pm(\A_{\wf_2,f}))$ constructed in Section \ref{twisted CM values},
let $W^\pm_f$ be the finite part of the Whittaker function of $\theta_{\wk_4,\chi} (\gb,\phi_f^\pm \phi_\infty^\pm)$.
For any $\alpha\in \wf_2^\times$, if $\alpha = \Nm \beta$ for some $\beta \in \A_{\wk_4,f}$, then
\begin{align*}
W^+_f(t(\alpha)) = \int_{\A_{\wk_4^1,f}} \phi_f^+(\alpha\beta^{-1} h) \chi(h) dh,\quad
W^-_f(t(\alpha)) = \int_{\A_{\wk_4^1,f}} \phi_f^-(\alpha\beta^{-1} h) \chi(h) dh
\end{align*}
by Equation \eqref{sl2 whittaker}. 
In this case, we have $-\varepsilon\alpha = \Nm (\beta\beta_0)$ and
$\phi_f^-(-\varepsilon\alpha (\beta\beta_0)^{-1} h) 
= \phi_f^+( \alpha \beta^{-1} h)$.
Therefore, 
\begin{align}\label{same finite whittaker}
W^+_f(t(\alpha \sqrt{\widetilde{D}})) = W^-_f( t(-\varepsilon \alpha \sqrt{\widetilde{D}})).
\end{align}

\begin{proof}[Proof of Theorem \ref{tsw introduction}]
Since $\A_f^\times = \Q \widehat{\Z}^\times$, for any $\hb\in \GL_2(\A_{\wf_2})$, we have $\nu(\hb) = q_\hb \xi_\hb$, where $q_\hb\in \Q,\xi_\hb\in\widehat{\Z}^\times$.
Making a change of variable $g'=d(q_\hb)^{-1} g d(q_\hb)$, for any $(\varphi,\Xi)\in \Sc(V)(\A)\times \Sc(U)(\A))$ , we can rewrite the integral $I(\hb,\varphi,\Xi,\rho)$ as
\begin{align*}
I(\hb,\varphi,\Xi,\rho) = \int_{[G]} \theta_V ( g' d(\xi_\hb),\hb,\varphi) 
\int_{F^1\backslash \A_{F^1}} \theta ( g' d(\xi_\hb), h\lambda,\Xi) \rho ( h\lambda ) dh dg'.
\end{align*}
Applying Equation \eqref{weil right}, we have
\begin{align*}
\Theta (\pi_\rho) = \{ I(\hb,\varphi,\Xi,\rho) : (\varphi,\Xi) \in (\Sc(V)(\A)\times \Sc(U)(\A))^{\SL_2(\widehat{\Z})} \}.
\end{align*}
For any $\mu\in L^\vee/L$, let $\phi_\mu^\pm \in \Sc(W^\pm(\A_{\wf_2,f}))$ be the Schwartz functions constructed from $\varphi_\mu$ as in Section \ref{twisted CM values},.
By Theorem \ref{twisted siegel weil} and Proposition \ref{match weight}, there exist Schwartz functions $(\varphi_\mu^\pm,\Xi_\mu^\pm) \in (\Sc(V)(\A_f)\times \Sc(U)(\A_f))^{\SL_2(\widehat{\Z})}$ such that 
\begin{align*}
\theta_{\wk_4,\chi} (\gb,\phi_\mu^\pm \phi_\infty^\pm ) =I (\hb,\varphi_\mu^\pm\varphi_\infty^{(\pm 1.\mp 1)},\Xi_\mu^\pm\Xi_\infty^-,\rho),\quad \gb=\hb \in \GL_2(\A_{\wf_2})
\end{align*}
for any $\mu\in L^\vee/L$. 
Due to Equation \eqref{same finite whittaker}, we may choose $(\varphi_\mu^+,\Xi_\mu^+ )= (\varphi_\mu^-,\Xi_\mu^-)$.
Since $\phi_\mu^\pm$ is $\Q$-valued, we can take $(\varphi_\mu^\pm,\Xi_\mu^\pm)$ to be $(T^{1,\Delta},\widetilde{\omega}_f)$-invariant as in \cite[Theorem 3.11]{BLY22}.
Therefore, applying Equation \eqref{CM values torus}, we have
\begin{align*}
\begin{split}
& \frac{4}{\deg Z(W)} \left[ \Phi_f^r\left(Z_\chi\left(W^+\right)\right)-(-1)^r \Phi_f^r\left(Z_\chi\left(W^-\right)\right) \right]  \\
=& (4\pi)^{-r} \int_{\SL_2(\Z)\backslash \h}^{\reg} \sum_{\mu\in L^\vee/L} R_\tau^r f_\mu (\tau) 
\left[ \theta_{\wk_4,\chi}(g_\tau^\Delta,\phi_\mu^+\phi_\infty^{1,-1}) 
- (-1)^r \theta_{\wk_4,\chi}(g_\tau^\Delta,\phi_\mu^+\phi_\infty^{-1,1}) \right] d\mu(\tau)\\
=&(-4\pi)^{-r} \int_{\SL_2(\Z)\backslash \h}^{\reg} \sum_{\mu\in L^\vee/L} f_\mu(\tau) \left(R_1+R_2\right)^r
\left( I(g_\tau^\Delta,\varphi_\mu^+ ( \varphi_\infty^{(1,-1)}-(-1)^r \varphi_\infty^{(-1,1)}),\Xi_\mu^+\Xi_\infty^-,\rho) \right) d\mu(\tau) \\
=& 2 \sqrt{D}^{r} \log \varepsilon_\rho \sum_{\mu\in L^\vee/L} c_\mu(f) 
- \sum_{m>0, \mu \in L^{\vee} / L} c(-m, \mu) \\
&\quad \times \sum_{\substack{\alpha \in \wf_2^+, \operatorname{Tr}(\alpha)=m \\ \exists \lambda_\alpha\in F, \text{ s.t. } \Nm \alpha = \Nm \lambda_\alpha^{-1} }} m^r P_r\left(\frac{\alpha-\alpha^{\prime}}{m}\right) 
\widetilde{W}_f(\alpha;\lambda_\alpha) \log |\lambda_\alpha/\lambda_\alpha'|,
\end{split}
\end{align*}
by Theorem \ref{twisted integral botcherds forms}.
Since  $(\varphi_\mu^+,\Xi_\mu^+) \in (\Sc(V)(\A_f)\times \Sc(U)(\A_f))^{\SL_2(\widehat{\Z})}$ is $(T^{1,\Delta},\widetilde{\omega}_f)$-invariant,
$c_\mu(f)\in \Q$ by \cite[Proposition 4.11]{BLY22}.
Then the claim follows from Remark \ref{comparing finite whittaker} and we are done.
\end{proof}

\begin{remark}
Comparing with the CM values in \cite{BLY22}, the condition about $\text{Diff}(W,t)$ is replaced by existence of $\lambda\in F$ such that $\Nm \lambda = \Nm\alpha^{-1}$ and $\beta \in \A_{\wk_4,f}$ such that $\Nm\beta = \alpha$. 
\end{remark}

\printbibliography

\end{document}